\documentclass[a4paper,12pt]{amsart}
\pdfoutput=1
\usepackage{amsfonts, amsmath, amssymb, amstext} 
\usepackage{tikz}
\usepackage[normalem]{ulem}
\usepackage{mathrsfs}
\usepackage{hyperref}
\usepackage{enumitem}

\usepackage[
backend=bibtex,
style=numeric,sortcites,
maxbibnames=99
]{biblatex}
\allowdisplaybreaks[2]

\numberwithin{equation}{section}

\theoremstyle{plain}
\newtheorem{thm}{Theorem}[section] 
\newtheorem{lem}[thm]{Lemma}
\newtheorem{prop}[thm]{Proposition}
\newtheorem{cor}[thm]{Corollary}

\theoremstyle{definition}
\newtheorem{defn}[thm]{Definition}

\theoremstyle{remark}

\newcommand{\norm}[1]{\left\|#1\right\|}
\newcommand{\cspan}{\overline{\mathrm{span}}}
\def\C{{\mathbb C}}
\def\N{{\mathbb N}}

\begin{document}

\title[A uniqueness theorem for Nica--Toeplitz algebras]{A uniqueness theorem for the Nica--Toeplitz algebra of a compactly aligned  product system}

\author{James Fletcher}
\address{School of Mathematics and Statistics \\ Victoria University of Wellington \\ Wellington 6140, New Zealand}
\email{james.fletcher@vuw.ac.nz}

\date{\today}
\thanks{This research was supported by an Australian Government Research Training Program (RTP) Scholarship and by the Marsden grant 15-UOO-071 from the Royal Society of New Zealand.}



\begin{abstract}
Fowler introduced the notion of a product system: a collection of Hilbert bimodules $\mathbf{X}=\left\{\mathbf{X}_p:p\in P\right\}$  indexed by a semigroup $P$, endowed with a multiplication implementing isomorphisms $\mathbf{X}_p\otimes_A \mathbf{X}_q\cong \mathbf{X}_{pq}$. When $P$ is quasi-lattice ordered, Fowler showed how to associate a $C^*$-algebra $\mathcal{NT}_\mathbf{X}$ to $\mathbf{X}$, generated by a universal representation satisfying some covariance condition. In this article we prove a uniqueness theorem for these so called Nica--Toeplitz algebras. 
\end{abstract}

\maketitle

\section{Introduction}

Suppose $A$ is a $C^*$-algebra and $X$ is a right Hilbert $A$-module. When $X$ comes equipped with a left action of $A$ by adjointable operators, we call $X$ a Hilbert $A$-bimodule. When this left action of $A$ is faithful, Pimsner showed how to associate a $C^*$-algebra $\mathcal{T}_X$ to $X$, called the Toeplitz algebra of $X$, generated by the raising and lowering operators on the Fock space of $X$ \cite{MR1426840}. In \cite{MR1722197}, Fowler and Raeburn generalised the situation to arbitrary left actions by adjointable operators, and showed that $\mathcal{T}_X$ may be realised as the universal $C^*$-algebra generated by so called Toeplitz representations. In the spirit of Coburn's Theorem for the classical Toeplitz algebra \cite{MR0213906}, Fowler and Raeburn also proved a uniqueness theorem for $\mathcal{T}_X$ \cite[Theorem~2.1]{MR1722197} that provides a sufficient condition for a representation of $\mathcal{T}_X$ to be faithful. Loosely speaking, their result states that a representation of $\mathcal{T}_X$ on a Hilbert space $\mathcal{H}$ will be faithful, provided the representation leaves enough room in $\mathcal{H}$ for the coefficient algebra $A$ to act faithfully. 

Subsequently Fowler introduced the notion of a product system of Hilbert bimodules \cite{MR1907896}, generalising the continuous product systems of Hilbert spaces studied by Arveson \cite{MR987590} and the discrete product systems studied by Dinh \cite{MR1138835}. Loosely speaking, a product system of Hilbert $A$-bimodules over a unital semigroup $P$ consists of a semigroup $\mathbf{X}=\bigsqcup_{p\in P} \mathbf{X}_p$, such that each $\mathbf{X}_p$ is a Hilbert $A$-bimodule, and the map $x\otimes_A y\mapsto xy$ extends to an isomorphism from $\mathbf{X}_p\otimes_A \mathbf{X}_q$ to $\mathbf{X}_{pq}$ for each $p,q\in P\setminus \{e\}$. Motivated by the work of Nica \cite{MR1241114} and Laca and Raeburn on Toeplitz algebras associated to non-abelian groups \cite{MR1402771}, Fowler studied representations of compactly aligned product systems over quasi-lattice ordered groups --- semigroups sitting inside groups possessing a semi-lattice like structure, satisfying an additional constraint called Nica covariance. Generalising the Toeplitz algebra associated to a single Hilbert bimodule, Fowler showed how to associate a $C^*$-algebra $\mathcal{NT}_\mathbf{X}$, generated by a universal Nica covariant representation, to each compactly aligned product system $\mathbf{X}$. We call this $C^*$-algebra the Nica--Toeplitz algebra of $\mathbf{X}$. Furthermore, Fowler associates a twisted semigroup crossed product algebra to each compactly aligned product system, and characterises their faithful representations \cite[Theorem~7.2]{MR1907896}. Restricting to the subalgebra $\mathcal{NT}_\mathbf{X}$ then gives a uniqueness theorem for Nica--Toeplitz algebras, generalising both Laca and Raeburn's uniqueness theorem for Toeplitz algebras of quasi-lattice ordered groups \cite[Theorem~3.7]{MR1402771} and Fowler and Raeburn's uniqueness theorem for Toeplitz algebras of Hilbert bimodules \cite[Theorem~2.1]{MR1722197} --- at least in the case where the bimodule is `essential'.  

In this article we prove a slightly more general version of Fowler's uniqueness theorem for Nica--Toeplitz algebras associated to compactly aligned product systems over quasi-lattice ordered groups. The result gives a sufficient condition for the induced representation $\psi_*$ of a Nica covariant representation $\psi$ (on a Hilbert space $\mathcal{H}$) to be faithful. This condition basically says that the ranges of all the operators $\left\{\psi_p(x): x\in \mathbf{X}_p, \ p\in P\setminus\{e\}\right\}$ should leave enough room in $\mathcal{H}$ for $A$ to act faithfully. When $A$ acts by compacts on each fibre of $\mathbf{X}$, we will see that this condition is also necessary. Unlike in Fowler's result, we do not insist that each fibre of the product system is essential (i.e. we do not require that $\mathbf{X}_p=\overline{A\cdot\mathbf{X}_p}$ for each $p\in P$). Whilst there do not seem to many `natural' examples of Hilbert bimodules with nondegenerate left actions, this level of generality was made use of in our article \cite{2017arXiv170608626F}. Moreover, in contrast to Fowler's proof, we do not view the Nica--Toeplitz algebra as a subalgebra of a twisted semigroup crossed product, working with $\mathcal{NT}_\mathbf{X}$ directly. Furthermore, we note that Fowler's result as stated in \cite[Theorem~7.2]{MR1907896} is technically false when applied to any product system over the trivial semigroup $\{e\}$, and correct this error in our result. 

The article is set out as follows. In Section~\ref{background} we fix notation and recap the necessary background material for Hilbert bimodules, product systems of Hilbert bimodules, and their associated Nica--Toeplitz algebras. In Section~\ref{proof of the uniqueness theorem} we present the proof of our uniqueness theorem for Nica-Toeplitz algebras.

After completing this article, it was brought to our attention that Kwa\'sniewski and Larsen had already proved a far-reaching generalisation of our (along with Fowler's) uniqueness theorem \cite[Corollary~10.14]{2016arXiv161108525K} for full Nica--Toeplitz algebras associated to well-aligned ideals of right tensor $C^*$-precategories. Subsequent to the initial version of this article appearing on the arXiv, Kwa\'sniewski and Larsen showed how their more general result can be applied to product systems over right LCM semigroups (themselves a generalisation of quasi-lattice ordered groups) \cite[Theorem~2.19]{2017arXiv170604951K}. Despite this, we still feel that the results in this article will find use amongst those studying product systems over quasi-lattice ordered groups. We provide a direct proof of the uniqueness theorem in the quasi-lattice ordered case, and as such avoid the various technical conditions present in \cite[Theorem~2.19]{2017arXiv170604951K}. We also note that in Corollary~\ref{representations in C*-algebras}, we show how to extend the uniqueneness theorem to representations in arbitrary $C^*$-algebras (rather than just on Hilbert spaces), provided the action on each fibre is compact. 

\section{Preliminaries}
\label{background}

\subsection{Hilbert bimodules}

We attempt to summarise only those aspects of Hilbert bimodules that we will need. Readers unfamiliar with this material, or looking for more detail, are directed to \cite{lance}.

Let $A$ be a $C^*$-algebra. A (right) inner-product $A$-module is a complex vector space $X$ equipped with a map $\langle \cdot, \cdot \rangle_A:X\times X\rightarrow A$, linear in its second argument, and a right action of $A$, such that for any $x,y\in X$ and $a\in A$, we have
\begin{enumerate}[label=\upshape(\roman*)]
\item $\langle x,y\rangle_A=\langle y,x\rangle_A^*$;
\item $\langle x,y\cdot a \rangle_A=\langle x,y\rangle_Aa$;
\item $\langle x,x\rangle_A\geq 0$ in $A$; and
\item $\langle x,x\rangle_A=0$ if and only if $x=0$.
\end{enumerate}
It follows from \cite[Proposition~1.1]{lance} that the formula $\norm{x}_X:=\norm{\langle x,x\rangle_A}_A^{1/2}$ defines a norm on $X$. If $X$ is complete with respect to this norm, we say that $X$ is a (right) Hilbert $A$-module. 

Let $X$ be a (right) Hilbert $A$-module. We say that a map $T:X\rightarrow X$ is adjointable if there exists a map $T^*:X\rightarrow X$ such that $\langle Tx, y\rangle_A=\langle x, T^*y\rangle_A$ for each $x,y\in X$. Every adjointable  operator $T$ is automatically linear and continuous, and the adjoint $T^*$ is unique. The collection of adjointable operators on $X$, denoted by $\mathcal{L}_A(X)$, equipped with the operator norm is a $C^*$-algebra. For each $x,y\in X$ there is an adjointable operator $\Theta_{x,y}\in \mathcal{L}_A(X)$ defined by $\Theta_{x,y}(z)=x\cdot \langle y, z\rangle_A$. We call operators of this form (generalised) rank-one operators. The closed subspace $\mathcal{K}_A(X):=\cspan\{\Theta_{x,y}:x,y\in X\}$ is an essential ideal of $\mathcal{L}_A(X)$, whose elements we refer to as (generalised) compact operators. 

A (right) Hilbert $B$--$A$-bimodule is a (right) Hilbert $A$-module $X$ together with a $*$-homomorphism $\phi:B\rightarrow \mathcal{L}_A(X)$. When $A=B$, we say that $X$ is a Hilbert $A$-bimodule. We think of $\phi$ as implementing a left action of $B$ on $X$, and frequently write $b\cdot x$ for $\phi(b)(x)$. Since each $\phi(b)\in \mathcal{L}_A(X)$ is $A$-linear, we have that $b\cdot (x\cdot a)=(b\cdot x)\cdot a$ for each $a\in A$, $b\in B$, and $x\in X$. 

An important example is the Hilbert $A$-bimodule ${}_A A_A$, which is just the set $A$ equipped with the inner product given by $\langle a,b\rangle_A=a^*b$ and left and right actions of $A$ given by multiplication. Then $\mathcal{K}_A({}_A A_A)$ is isomorphic to $A$ via the map $\Theta_{a,b}\mapsto ab^*$, whilst $\mathcal{L}_A({}_A A_A)$ is isomorphic to the multiplier algebra of $A$. 

The Hewitt--Cohen--Blanchard factorisation theorem \cite[Proposition~2.31]{MR1634408} says that if $x$ is an element of a Hilbert $A$-module $X$, then there exists a unique $x'\in X$ such that $x=x'\cdot \langle x', x'\rangle_A$. Hence, $X$ is a right nondegenerate $A$-module in the sense that $X=\cspan\{x\cdot a:x\in X,a\in A\}$. It is not necessarily true that every Hilbert $A$-bimodule is left nondegenerate in the sense that $X=\cspan\{a\cdot x:x\in X,a\in A\}$ (in \cite{MR1907896}, Fowler calls such bimodules  essential). 

The balanced tensor product $X\otimes_A Y$ of two Hilbert $A$-bimodules $X$ and $Y$ is formed as follows. Let $X \odot Y$ be the algebraic tensor product of $X$ and $Y$ as complex vector spaces, and let $X\odot_A Y$ be the quotient of $X \odot Y$ by the subspace spanned by elements of the form $x\cdot a \odot y- x\odot a\cdot y$ where $x\in X$, $y\in Y$, and $a\in A$ (we write $x\odot_A y$ for the coset containing $x\odot y$). Then the formula
$
\langle x\odot_A y, w\odot_A z\rangle_A:=\langle y, \langle x,w\rangle_A \cdot z\rangle_A,
$
determines a bounded $A$-valued sesquilinear form on $X\odot_A Y$. Let $N$ be the subspace $\mathrm{span}\{n\in X\odot_A Y:\langle n,n\rangle_A=0\}$. The formula  $\norm{z+N}:=\inf_{n\in N} \norm{\langle z+n,z+n\rangle_A}_A^{1/2}$ defines a norm on $(X\odot_A Y)/N$, and we define $X\otimes_A Y$ to be the completion of $(X\odot_A Y)/N$ with respect to this norm. The balanced tensor product $X\otimes_A Y$ carries a left and right action of $A$, such that $a\cdot (x\otimes_A y)\cdot b=(a\cdot x)\otimes_A (y\cdot b)$ for each $x\in X$, $y\in Y$, and $a,b\in A$. 

Given Hilbert $A$-bimodules $X$ and $Y$ and an adjointable operator $S\in \mathcal{L}_A(X)$, there exists an adjointable operator $S\otimes_A \mathrm{id}_Y\in \mathcal{L}_A(X\otimes_A Y)$ (with adjoint $S^*\otimes_A \mathrm{id}_Y$) determined by the formula $(S\otimes_A \mathrm{id}_Y)(x\otimes_A y)= (Sx)\otimes_A y$ for each $x\in X$ and $y\in Y$. 

We will also make use of the theory of induced representations. Given a Hilbert $B$--$A$-bimodule $X$ and a nondegenerate representation $\pi:A\rightarrow \mathcal{B}(\mathcal{H})$ of $A$ on a Hilbert space $\mathcal{H}$, \cite[Proposition~2.66]{MR1634408} gives a representation $X\text{-}\mathrm{Ind}_A^B\pi:B\rightarrow \mathcal{B}\left(X\otimes_A \mathcal{H}\right)$ such that $\left(X\text{-}\mathrm{Ind}_A^B\pi\right)(b)(x\otimes_A h)=(b\cdot x)\otimes_A h$ for each $b\in B$, $x\in X$, and $h\in \mathcal{H}$. 

\subsection{Product systems of Hilbert bimodules and quasi-lattice ordered groups}

Let $A$ be a $C^*$-algebra and $P$ a semigroup with identity $e$. A product system over $P$ with coefficient algebra $A$ is a semigroup $\mathbf{X}=\bigsqcup_{p\in P}\mathbf{X}_p$ such that
\begin{enumerate}[label=\upshape(\roman*)]
\item $\mathbf{X}_p\subseteq \mathbf{X}$ is a Hilbert $A$-bimodule for each $p\in P$;
\item $\mathbf{X}_e$ is equal to the Hilbert $A$-bimodule ${}_A A_A$;
\item For each $p,q\in P\setminus \{e\}$, there exists a Hilbert $A$-bimodule isomorphism $M_{p,q}:\mathbf{X}_p\otimes_A \mathbf{X}_q\rightarrow \mathbf{X}_{pq}$ satisfying $M_{p,q}(x\otimes_A y)=xy$ for each $x\in \mathbf{X}_p$ and $y\in \mathbf{X}_q$;
\item[(iv)] Multiplication in $\mathbf{X}$ by elements of $\mathbf{X}_e=A$ implements the left and right actions of $A$ on each $\mathbf{X}_p$, i.e. $xa=x\cdot a$ and $ax=a\cdot x$ for each $a\in A$, $x\in \mathbf{X}_p$, and $p\in P$. 
\end{enumerate}

For each $p\in P$, we write $\phi_p:A\rightarrow \mathcal{L}_A(\mathbf{X}_p)$ for the $*$-homomorphism that implements the left action of $A$ on $\mathbf{X}_p$, i.e. $\phi_p(a)(x)=a\cdot x=ax$ for each $a\in A$ and $x\in \mathbf{X}_p$. Since $\mathbf{X}$ is a semigroup, multiplication in $\mathbf{X}$ is associative. In particular, $\phi_{pq}(a)(xy)=(\phi_p(a)x)y$ for all $p,q\in P$, $a\in A$, $x\in \mathbf{X}_p$, and $y\in \mathbf{X}_q$. Also, for each $p\in P$, we write $\langle \cdot,\cdot\rangle_A^p$ for the $A$-valued inner-product on $\mathbf{X}_p$. 

By (ii) and (iv), for each $p\in P$ there exist $A$-linear inner-product preserving maps $M_{p,e}:\mathbf{X}_p\otimes_A \mathbf{X}_e \rightarrow \mathbf{X}_p$ and $M_{e,p}:\mathbf{X}_e\otimes_A \mathbf{X}_p\rightarrow \mathbf{X}_p$ such that $M_{p,e}(x\otimes_A a)=xa=x\cdot a$ and $M_{e,p}(a\otimes_A x)=ax=a\cdot x$ for each $a\in \mathbf{X}_e=A$ and $x\in \mathbf{X}_p$. By the Hewitt--Cohen--Blanchard factorisation theorem, each $M_{p,e}$ is automatically an $A$-bimodule isomorphism. On the other hand, the maps $M_{e,p}$ need not be isomorphisms, since we do not require that each $\mathbf{X}_p$ is (left) nondegenerate (i.e. $M_{e,p}$ need not be surjective). 

Given $p\in P\setminus \{e\}$ and $q\in P$, the $A$-bimodule isomorphism $M_{p,q}:\mathbf{X}_p\otimes_A \mathbf{X}_q\rightarrow \mathbf{X}_{pq}$ enables us to define a $*$-homomorphism $\iota_p^{pq}:\mathcal{L}_A\left(\mathbf{X}_p\right)\rightarrow \mathcal{L}_A\left(\mathbf{X}_{pq}\right)$ by 
\[
\iota_p^{pq}(S):=M_{p,q}\circ \left(S\otimes_A \mathrm{id}_{\mathbf{X}_q}\right)\circ M_{p,q}^{-1}
\]
for each $S\in \mathcal{L}_A\left(\mathbf{X}_p\right)$. Equivalently, the $*$-homomorphism $\iota_p^{pq}$ is characterised by the formula $\iota_p^{pq}(S)(xy)=(Sx)y$ for each $S\in \mathcal{L}_A\left(\mathbf{X}_p\right)$, $x\in \mathbf{X}_p$, and $y\in \mathbf{X}_q$. 
Since $\mathbf{X}_e \otimes_A \mathbf{X}_q$ need not in general be isomorphic to $\mathbf{X}_q$, we cannot always define a map from $\mathcal{L}_A\left(\mathbf{X}_e\right)$ to $\mathcal{L}_A\left(\mathbf{X}_q\right)$ using the above procedure. However, as $\mathcal{K}_A\left(\mathbf{X}_e\right)=\mathcal{K}_A\left({}_A A_A\right)\cong A$, we can define $\iota_e^q:\mathcal{K}_A\left(\mathbf{X}_e\right)\rightarrow \mathcal{L}_A\left(\mathbf{X}_q\right)$ by $\iota_e^q(a):=\phi_q(a)$. For notational purposes, we define $\iota_p^r: \mathcal{L}_A\left(\mathbf{X}_p\right)\rightarrow \mathcal{L}_A\left(\mathbf{X}_r\right)$ to be the zero map whenever $p,r\in P$ and $r\neq pq$ for all $q\in P$.

We are primarily interested in situations where the underlying semigroup possesses some additional order structure. In particular we focus on the quasi-lattice ordered groups introduced by Nica \cite{MR1241114}. A quasi-lattice ordered group $(G,P)$ consists of a group $G$ and a subsemigroup $P$ of $G$ such that $P\cap P^{-1}=\{e\}$, and with respect to the partial order on $G$ induced by $p\leq q \Leftrightarrow p^{-1}q\in P$, any two elements $p,q\in G$ which have a common upper bound in $P$ have a least common upper bound in $P$. It is straightforward to show that if two elements in $G$ have a least common upper bound in $P$, then this least common upper bound is unique. We write $p\vee q$ for the least common upper bound of $p,q\in G$ if it exists. For $p,q\in G$, we write $p\vee q=\infty$ if $p$ and $q$ have no common upper bound in $P$, and $p\vee q<\infty$ otherwise. We can also extend the notion of least upper bounds in $(G,P)$ from pairs of elements in $P$ to finite subsets of $P$. We define $\bigvee\emptyset:=e$, $\bigvee \{p\}:=p$ for any $p\in P$, and for any $n \geq 2$ and $C:=\{p_1,\ldots p_n\}\subseteq P$ we define $\bigvee C:=p_1\vee \cdots \vee p_n$ (since $P\cap P^{-1}=\{e\}$, the relation $\leq$ is antisymmetric, and so this is well-defined). 

Let $(G,P)$ be a quasi-lattice ordered group and $\mathbf{X}$ a product system over $P$. We say that $\mathbf{X}$ is compactly aligned if, whenever $S\in \mathcal{K}_A(\mathbf{X}_p)$ and $T\in \mathcal{K}_A(\mathbf{X}_p)$ for some $p,q\in P$ with $p\vee q<\infty$, we have $\iota_p^{p\vee q}(S)\iota_q^{p\vee q}(T)\in \mathcal{K}_A(\mathbf{X}_{p\vee q})$. Note that this condition does not imply that either $\iota_p^{p\vee q}(S)$ or $\iota_q^{p\vee q}(T)$ is compact. 

\subsection{Representations of compactly aligned product systems, Nica covariance, and the Nica--Toeplitz algebra}

Let $(G,P)$ be a quasi-lattice ordered group and $\mathbf{X}$ a compactly aligned product system over $P$ with coefficient algebra $A$. A representation of $\mathbf{X}$ in a $C^*$-algebra $B$ is a map $\psi:\mathbf{X}\rightarrow B$ such that:
\begin{enumerate}
\item[(T1)] each $\psi_p:=\psi|_{\mathbf{X}_p}$ is a linear map, and $\psi_e$ is a $C^*$-homomorphism;
\item[(T2)] $\psi_p(x)\psi_q(y)=\psi_{pq}(xy)$ for all $p,q\in P$ and $x\in \mathbf{X}_p$, $y\in \mathbf{X}_q$;
\item[(T3)] $\psi_p(x)^*\psi_p(y)=\psi_e(\langle x,y\rangle_A^p)$ for all $p\in P$ and $x,y\in \mathbf{X}_p$.
\end{enumerate}

It follows from (T1) and (T3) that a representation $\psi$ is always norm-decreasing, and isometric if and only $\psi_e$ is injective. Proposition~8.11 of \cite{MR2135030} shows that for each $p\in P$, there exists a $*$-homomorphism $\psi^{(p)}:\mathcal{K}_A\left(\mathbf{X}_p\right)\rightarrow B$ such that $\psi^{(p)}\left(\Theta_{x,y}\right)=\psi_p(x)\psi_p(y)^*$ for all $x,y\in \mathbf{X}_p$. 

We say that a representation $\psi:\mathbf{X}\rightarrow B$ is Nica covariant if, for any $p,q\in P$ and $S\in \mathcal{K}_A(\mathbf{X}_p)$, $T\in \mathcal{K}_A(\mathbf{X}_q)$, we have
\begin{align*}
\psi^{(p)}(S)\psi^{(q)}(T)=
\begin{cases}
\psi^{(p\vee q)}\left(\iota_p^{p\vee q}(S)\iota_q^{p\vee q}(T)\right) & \text{if $p\vee q<\infty$}\\
0 & \text{otherwise.}
\end{cases}
\end{align*}
It follows from an application of the Hewitt--Cohen--Blanchard factorisation theorem that for any $p,q\in P$, we have
\[
\psi_p(\mathbf{X}_p)^*\psi_q(\mathbf{X}_q)\in 
\begin{cases}
\cspan\{\psi_{p^{-1}(p\vee q)}(\mathbf{X}_{p^{-1}(p\vee q)})\psi_{q^{-1}(p\vee q)}(\mathbf{X}_{q^{-1}(p\vee q)})^*\} & \text{if $p\vee q<\infty$}\\
\{0\} & \text{if $p\vee q=\infty$.}
\end{cases}
\]

Associated to each product system there exists a canonical Nica covariant representation called the Fock representation. We let $\mathcal{F}_\mathbf{X}:=\bigoplus_{p\in P} \mathbf{X}_p$ denote the space of sequences $(x_p)_{p\in P}$ such that $x_p\in \mathbf{X}_p$ for each $p\in P$ and $\sum_{p\in P}\langle x_p,x_p\rangle_A^p$ converges in $A$. By \cite[Proposition~1.1]{lance} there exists a well defined $A$-valued inner product on $\mathcal{F}_\mathbf{X}$ such that $\left\langle (x_p)_{p\in P}, (y_p)_{p\in P}\right\rangle_A=\sum_{p\in P}\langle x_p, y_p\rangle_A^p$, and that $\mathcal{F}_\mathbf{X}$ is complete with respect to the induced norm. Letting $A$ act pointwise from the left and right gives $\mathcal{F}_\mathbf{X}$ the structure of a Hilbert $A$-bimodule, which we call the Fock space of $\mathbf{X}$. Lemma~5.3 of \cite{MR1907896} shows that there exists an isometric Nica covariant representation $l:\mathbf{X}\rightarrow \mathcal{L}_A(\mathcal{F}_\mathbf{X})$ such that $l_p(x)(y_q)_{q\in P}=\left(xy_q\right)_{q\in P}$ for each $p\in P$, $x\in \mathbf{X}_p$, and $(y_q)_{q\in P}\in \mathcal{F}_\mathbf{X}$. We call $l$ the Fock representation of $\mathbf{X}$. 

Using \cite[Theorem~2.10]{MR2679392} it can be shown that there exists a $C^*$-algebra $\mathcal{NT}_\mathbf{X}$, which we call the Nica--Toeplitz algebra of $\mathbf{X}$, and a Nica covariant representation $i_\mathbf{X}:\mathbf{X}\rightarrow \mathcal{NT}_\mathbf{X}$, that are universal in the following sense:
\begin{enumerate}[label=\upshape(\roman*)]
\item $\mathcal{NT}_\mathbf{X}$ is generated by the image of $i_\mathbf{X}$;
\item if $\psi:\mathbf{X}\rightarrow B$ is any other Nica covariant representation of $\mathbf{X}$, then there exists a $*$-homomorphism $\psi_*:\mathcal{NT}_\mathbf{X}\rightarrow B$ such that $\psi_* \circ i_\mathbf{X}=\psi$. 
\end{enumerate}
Since $i_\mathbf{X}$ generates $\mathcal{NT}_\mathbf{X}$, it follows that  $\mathcal{NT}_\mathbf{X}=\cspan\left\{i_\mathbf{X}(x)i_\mathbf{X}(y)^*:x,y\in \mathbf{X}\right\}$.

Proposition~4.7 of \cite{MR1907896} shows that there exists a coaction $\delta_\mathbf{X}:\mathcal{NT}_\mathbf{X}\rightarrow \mathcal{NT}_\mathbf{X}\otimes C^*(G)$ (we use an unadorned $\otimes$ to denote the minimal tensor product of $C^*$-algebras), which we call the canonical gauge coaction, such that $\delta_\mathbf{X}(i_{\mathbf{X}_p}(x))=i_{\mathbf{X}_p}(x)\otimes i_G(p)$ for each $p\in P$ and $x\in \mathbf{X}_p$. For those readers interested in learning more about coactions in general, we suggest \cite[Appendix~A]{MR2203930}. 

Lemma~1.3 of \cite{MR1375586} shows that there exists a conditional expectation $E_{\delta_\mathbf{X}}$ of $\mathcal{NT}_\mathbf{X}$ onto the generalised fixed-point algebra $\mathcal{NT}_\mathbf{X}^{\delta_\mathbf{X}}:=\{b\in \mathcal{NT}_\mathbf{X}:\delta_\mathbf{X}(b)=b\otimes i_G(e)\}$ defined by $E_{\delta_\mathbf{X}}:=(\mathrm{id}_{ \mathcal{NT}_\mathbf{X}}\otimes \rho)\circ \delta_\mathbf{X}$, where $\rho:C^*(G)\rightarrow \C$ is the canonical trace. It can be shown that $\mathcal{NT}_\mathbf{X}^{\delta_\mathbf{X}}=\cspan\{i_{\mathbf{X}_p}(\mathbf{X}_p)i_{\mathbf{X}_p}(\mathbf{X}_p)^*:p\in P\}$, and, for any $p,q\in P$, $x\in \mathbf{X}_p$, $y\in \mathbf{X}_q$, we have $E_{\delta_\mathbf{X}}\left(i_{\mathbf{X}_p}(x)i_{\mathbf{X}_q}(y)^*\right)=\delta_{p,q}i_{\mathbf{X}_p}(x)i_{\mathbf{X}_q}(y)^*$. We are particularly interested in the situation where the expectation $E_{\delta_\mathbf{X}}$ is faithful on positive elements, i.e. $E_{\delta_\mathbf{X}}(b^*b)=0 \Rightarrow b=0$ for any $b\in \mathcal{NT}_\mathbf{X}$. Inspired by \cite[Definition~7.1]{MR1907896}, we say that a compactly aligned product system $\mathbf{X}$ is amenable if $E_{\delta_\mathbf{X}}$ is faithful on positive elements. The argument of \cite[Lemma~6.5]{MR1402771} shows that if $G$ is an amenable group, then $\mathbf{X}$ is an amenable product system. 

\section{A uniqueness theorem for Nica--Toeplitz algebras}
\label{proof of the uniqueness theorem}

Firstly, we fix some notation. 

\begin{defn}
\label{basic projections}
Let $(G,P)$ be a quasi-lattice ordered group, $\mathbf{X}$ a compactly aligned product system over $P$ with coefficient algebra $A$, and $\psi:\mathbf{X}\rightarrow \mathcal{B}(\mathcal{H})$ a Nica covariant representation of $\mathbf{X}$ on a Hilbert space $\mathcal{H}$. We define a collection $\{P_p^\psi:p\in P\}$ of projections in $\mathcal{B}(\mathcal{H})$ by $P_e^\psi:=\mathrm{id}_\mathcal{H}$ and $P_p^\psi:=\mathrm{proj}_{\overline{\psi_p(\mathbf{X}_p)\mathcal{H}}}$ for each $p\in P\setminus\{e\}$. We also set $P_\infty^\psi:=0$.
\end{defn}

The purpose of this article is to prove the following result:

\begin{thm}
\label{uniqueness theorem for NT algebras}
Let $(G,P)$ be a quasi-lattice ordered group and $\mathbf{X}$ a compactly aligned product system over $P$ with coefficient algebra $A$. Suppose $\psi:\mathbf{X}\rightarrow \mathcal{B}(\mathcal{H})$ is a Nica covariant representation of $\mathbf{X}$ on a Hilbert space $\mathcal{H}$. 
\begin{enumerate}[label=\upshape(\roman*)]
\item
If the product system $\mathbf{X}$ is amenable, and, for any finite set $K\subseteq P\setminus \{e\}$, the representation
\begin{align*}
A\ni a \mapsto \psi_e(a)\prod_{t\in K}\big(\mathrm{id}_\mathcal{H}-P_t^\psi\big)\in \mathcal{B}(\mathcal{H})
\end{align*}
is faithful, then the induced $*$-homomorphism $\psi_*:\mathcal{NT}_\mathbf{X}\rightarrow \mathcal{B}(\mathcal{H})$ is faithful.
\item
If $\psi_*$ is faithful and $\phi_p(A)\subseteq \mathcal{K}_A(\mathbf{X}_p)$ for each $p\in P$, then the representation 
\begin{align*}
A\ni a \mapsto \psi_e(a)\prod_{t\in K}\big(\mathrm{id}_\mathcal{H}-P_t^\psi\big)\in \mathcal{B}(\mathcal{H})
\end{align*}
is faithful for any finite set $K\subseteq P\setminus \{e\}$.
\end{enumerate}
\end{thm}

The main step in the proof of the uniqueness theorem is to show that the expectation $E_{\delta_\mathbf{X}}$ is also implemented spatially, i.e. there is a compatible expectation $E_\psi$ of $\psi_*(\mathcal{NT}_\mathbf{X})$ onto $\psi_*(\mathcal{NT}_\mathbf{X}^{\delta_\mathbf{X}})$. To get this compatible expectation we need to be able to calculate the norms of elements of $\psi_*(\mathcal{NT}_\mathbf{X}^{\delta_\mathbf{X}})$. To do this we will make use of the following well-known fact about operators on Hilbert spaces: if $P_1,\ldots, P_n\in \mathcal{B}(\mathcal{H})$ are mutually orthogonal projections that commute with $T\in \mathcal{B}(\mathcal{H})$ and satisfy $\sum_{i=1}^n P_i=\mathrm{id}_\mathcal{H}$, then $\norm{T}_{\mathcal{B}(\mathcal{H})}=\max_{1\leq i\leq n}\norm{P_i T}_{\mathcal{B}(\mathcal{H})}$. 

We now work towards showing that there exists a collection of mutually orthogonal projections in $\mathcal{B}(\mathcal{H})$ that decompose the identity and commute with everything in $\psi_*(\mathcal{NT}_\mathbf{X}^{\delta_\mathbf{X}})$. 

We begin by showing that the $*$-homomorphism $\psi^{(p)}:\mathcal{K}_A(\mathbf{X}_p)\rightarrow \mathcal{B}(\mathcal{H})$ has a canonical extension to all of $\mathcal{L}_A(\mathbf{X}_p)$ (for each $p\in P$), and establish some properties of this extension. 

\begin{prop}
\label{existence of rho map}
Let $(G,P)$ be a quasi-lattice ordered group and $\mathbf{X}$ a compactly aligned product system over $P$ with coefficient algebra $A$. Let $\psi:\mathbf{X}\rightarrow \mathcal{B}(\mathcal{H})$ be a Nica covariant representation of $\mathbf{X}$ on a Hilbert space $\mathcal{H}$. Then
\begin{enumerate}[label=\upshape(\roman*)]
\item 
For each $p\in P$, there exists a representation $\rho_p^\psi:\mathcal{L}_A\left(\mathbf{X}_p\right)\rightarrow \mathcal{B}(\mathcal{H})$ such that for each $S\in \mathcal{L}_A(\mathbf{X}_p)$,
\[
\rho_p^\psi(S)\left(\psi_p(x)h\right)=\psi_p(Sx)h \quad \text{for each $x\in \mathbf{X}_p$, $h\in \mathcal{H}$}
\]
and $\rho_p^\psi(S)$ is zero on $\left(\psi_p\left(\mathbf{X}_p\right)\mathcal{H}\right)^\perp$.
\item 
$\rho_p^\psi|_{\mathcal{K}_A\left(\mathbf{X}_p\right)}=\psi^{(p)}$.
\item
For any $q\in P$ and $a\in A\cong \mathcal{K}_A(\mathbf{X}_e)$, we have 
$\rho_{q}^\psi\left(\iota_e^{q}(a)\right)=\rho_{e}^\psi(a)P_{q}^\psi$.
Furthermore, if  $p\in P\setminus \{e\}$, then
$\rho_{pq}^\psi\left(\iota_p^{pq}(S)\right)=\rho_{p}^\psi(S)P_{pq}^\psi$
for any and $S\in \mathcal{L}_A(\mathbf{X}_p)$. 
\item
If $\mathcal{K}\subseteq \mathcal{H}$ is a $\psi_e$-invariant subspace of $\mathcal{H}$, then the subspace $\mathcal{M}:=\overline{\psi_p\left(\mathbf{X}_p\right)\mathcal{K}}$ is $\rho_p^\psi$-invariant. Furthermore, if $\psi_e|_\mathcal{K}$ is faithful, then $\rho_p^\psi|_\mathcal{M}$ is also faithful. 
\end{enumerate}
\begin{proof}
Observe that for any $p\in P$ and $x,y\in \mathbf{X}_p$ and $h,k\in \mathcal{H}$, we have
\begin{equation}
\label{existence of U}
\begin{aligned}
\langle \psi_p(x)h,\psi_p(y)k\rangle_\C
\hspace{-0.1em}=\hspace{-0.1em}
\langle h, \psi_p(x)^*\psi_p(y)k\rangle_\C
\hspace{-0.1em}=\hspace{-0.1em}
\langle h, \psi_e\hspace{-0.1em}\left(\hspace{-0.1em}\langle x, y \rangle_A\hspace{-0.1em}\right)k\rangle_\C
&\hspace{-0.1em}=\hspace{-0.1em}
\langle x\hspace{-0.1em}\otimes_A \hspace{-0.1em}h, y\hspace{-0.1em}\otimes_A\hspace{-0.1em} k\rangle_\C. 
\end{aligned}
\end{equation}
Thus, there exists a linear isometry $U:\mathbf{X}_p \otimes_A \mathcal{H}\rightarrow \mathcal{H}$ such that 
$U\left(x\otimes_A h\right)=\psi_p(x)h$
for each $x\in \mathbf{X}_p$ and $h\in \mathcal{H}$. Equation~\eqref{existence of U} shows that $U^*\left(\psi_p(x)h\right)=x\otimes_A h$ for each $x\in \mathbf{X}_p$ and $h\in \mathcal{H}$. We claim that $U^*|_{\left(\psi_p\left(\mathbf{X}_p\right)\mathcal{H}\right)^\perp}=0$. To see this, observe that for any $f\in \left(\psi_p\left(\mathbf{X}_p\right)\mathcal{H}\right)^\perp$ and $y\in \mathbf{X}_p$, $h\in \mathcal{H}$ we have
\begin{align*}
\langle U(y\otimes_A k),f\rangle_\C=\langle \psi_p(y)k, f\rangle_\C=0,
\end{align*}
and hence $U^*(f)=0$. Since 
$
\mathbf{X}_p\otimes_A \mathcal{H}=\left(\mathbf{X}_p\cdot A\right)\otimes_A \mathcal{H}=\mathbf{X}_p\otimes_A \overline{\psi_e(A)\mathcal{H}},
$
we may assume that the representation $\psi_e$ is nondegenerate without loss of generality. With this in mind, define $\rho_p^\psi:\mathcal{L}_A(\mathbf{X}_p)\rightarrow \mathcal{B}(\mathcal{H})$ by
\[
\rho_p^\psi(S):=U\circ \left(\mathbf{X}_p\text{-}\mathrm{Ind}_A^{\mathcal{L}_A\left(\mathbf{X}_p\right)}\psi_e(S)\right)\circ U^*
\] 
for each $S\in \mathcal{L}_A\left(\mathbf{X}_p\right)$. Thus, for each $S\in \mathcal{L}_A\left(\mathbf{X}_p\right)$, the restriction $\rho_p^\psi(S)|_{\left(\psi_p\left(\mathbf{X}_p\right)\mathcal{H}\right)^\perp}$ is zero, whilst for any $x\in \mathbf{X}_p$ and $h\in \mathcal{H}$ we have 
\begin{align*}
\rho_p^\psi(S)(\psi_p(x)h)
&=\left(U\circ \left(\mathbf{X}_p\text{-}\mathrm{Ind}_A^{\mathcal{L}_A\left(\mathbf{X}_p\right)}\psi_e(S)\right)\circ U^*\right)\left(\psi_p(x)h\right)\\
&=\left(U\circ\left(\mathbf{X}_p\text{-}\mathrm{Ind}_A^{\mathcal{L}_A\left(\mathbf{X}_p\right)}\psi_e(S)\right)\right)(x\otimes_A h)
=U\left(Sx\otimes_A h\right)
=\psi_p(Sx)h. 
\end{align*}
This completes the proof of part (i).

Since both $\psi^{(p)}$ and $\rho_p^\psi$ are $*$-homomorphisms, to prove (ii) it suffices to show that $\psi^{(p)}$ and $\rho_p^\psi$ agree on rank-one operators. Fix $x,y \in \mathbf{X}_p$.  Firstly, we check that $\psi^{(p)}\left(\Theta_{x,y}\right)$ and $\rho_p^\psi\left(\Theta_{x,y}\right)$ agree on $\overline{\psi_p\left(\mathbf{X}_p\right)\mathcal{H}}$. For any $z\in \mathbf{X}_p$ and $h\in \mathcal{H}$, we have
\begin{align*}
\rho_p^\psi\left(\Theta_{x,y}\right)\left(\psi_p(z)h\right)
&=\psi_p\left(\Theta_{x,y}(z)\right)h
=\psi_p\left(x\cdot \langle y,z \rangle_A\right)h\\
&=\psi_p(x)\psi_p(y)^*\psi_p(z)h
=\psi^{(p)}\left(\Theta_{x,y}\right)\left(\psi_p(z)h\right).
\end{align*}
Since both $\psi^{(p)}\left(\Theta_{x,y}\right)$ and $\rho_p^\psi\left(\Theta_{x,y}\right)$ are linear and continuous, we conclude that they agree on $\overline{\psi_p\left(\mathbf{X}_p\right)\mathcal{H}}$. 
It remains to check that $\psi^{(p)}\left(\Theta_{x,y}\right)$ and $\rho_p^\psi\left(\Theta_{x,y}\right)$ agree on the orthogonal complement $\left(\psi_p\left(\mathbf{X}_p\right)\mathcal{H}\right)^\perp$. Making use of part (i), we see that the restriction $\rho_p^\psi\left(\Theta_{x,y}\right)|_{\left(\psi_p\left(\mathbf{X}_p\right)\mathcal{H}\right)^\perp}=0$. Since 
\[
\big\langle \psi^{(p)}\left(\Theta_{x,y}\right)h,k\big\rangle_\C=\langle \psi_p(x)\psi_p(y)^*h,k\rangle_\C=\langle h, \psi_p(y)\psi_p(x)^*k\rangle_\C=0
\]
for any $h\in \left(\psi_p\left(\mathbf{X}_p\right)\mathcal{H}\right)^\perp$ and $k\in \mathcal{H}$, we conclude that $\psi^{(p)}(\Theta_{x,y})|_{\left(\psi_p\left(\mathbf{X}_p\right)\mathcal{H}\right)^\perp}=0$ as well. This completes the proof of part (ii).

We now prove part (iii). Let $q\in P$ and $a\in A$. If $q=e$, then $P_q^\psi=\mathrm{id}_\mathcal{H}$, and so
\[
\rho_e^\psi(a)P_q=\rho_e^\psi(a)=\rho_e^\psi\left(\iota_e^e(a)\right)=\rho_e^\psi\left(\iota_e^q(a)\right). 
\]
On the other hand, if $q\neq e$, then $P_q^\psi=\mathrm{proj}_{\overline{\psi_q\left(\mathbf{X}_q\right)\mathcal{H}}}$. Hence, both $\rho_e^\psi(a)P_q$ and $\rho_e^\psi(\iota_e^q(a))$ are zero on $\left(\psi_q\left(\mathbf{X}_q\right)\mathcal{H}\right)^\perp$. Since $\rho_q^\psi\left(\iota_e^q(a)\right)$ and $\rho_e^\psi(a)P_q^\psi$ are linear and continuous, whilst 
\begin{align*}
\rho_q^\psi\left(\iota_e^q(a)\right)\left(\psi_q(x)h\right)
&=\psi_q\left(\iota_e^q(a)x\right)h
=\psi_q\left(a\cdot x\right)h\\
&=\rho_e^\psi(a)\left(\psi_q(x)h\right)
=\rho_e^\psi(a)P_q^\psi\left(\psi_q(x)h\right)
\end{align*}
for any $x\in \mathbf{X}_q$ and $h\in \mathcal{H}$, we see that $\rho_q^\psi\left(\iota_e^q(a)\right)$ and $\rho_e^\psi(a)P_q^\psi$ also agree on $ \overline{\psi_q\left(\mathbf{X}_q\right)\mathcal{H}}$. Thus, $\rho_q^\psi\left(\iota_e^q(a)\right)=\rho_e^\psi(a)P_q^\psi$. 

Now fix $p\in P\setminus \{e\}$ and $S\in \mathcal{L}_A\left(\mathbf{X}_p\right)$. Since $pq\neq e$, both $\rho_{pq}^\psi\left(\iota_p^{pq}(S)\right)$ and $\rho_{p}^\psi(S)P_{pq}^\psi$ are zero on the orthogonal complement $\left(\psi_{pq}\left(\mathbf{X}_{pq}\right)\mathcal{H}\right)^\perp$. Observe that for any $x\in \mathbf{X}_p$, $y\in \mathbf{X}_q$, and $h\in \mathcal{H}$, we have
\begin{align*}
\rho_{pq}^\psi\left(\iota_p^{pq}(S)\right)\left(\psi_{pq}(xy)h\right)
=\psi_{pq}\left(\iota_p^{pq}(S)(xy)\right)h
=\psi_{pq}\left((Sx)y\right)h
=\psi_p(Sx)\psi_q(y)h,
\end{align*}
which by part (i) is the same as
\begin{align*}
\rho_p^\psi(S)\left(\psi_p(x)\psi_q(y)h\right)
=\rho_p^\psi(S)\left(\psi_{pq}(xy)h\right)
=\rho_p^\psi(S)P_{pq}^\psi\left(\psi_{pq}(xy)h\right).
\end{align*}
Since $\overline{\psi_{pq}\left(\mathbf{X}_{pq}\right)\mathcal{H}}=\overline{\psi_p\left(\mathbf{X}_p\right)\psi_q\left(\mathbf{X}_q\right)\mathcal{H}}$, whilst $\rho_{pq}^\psi\left(\iota_p^{pq}(S)\right)$ and $\rho_p^\psi(S)P_{pq}^\psi$ are linear and continuous, we conclude that $\rho_{pq}^\psi\left(\iota_p^{pq}(S)\right)$ and $\rho_p^\psi(S)P_{pq}^\psi$ are also equal on $\overline{\psi_{pq}\left(\mathbf{X}_{pq}\right)\mathcal{H}}$. 

Finally, we prove part (iv). Firstly, observe that the subspace $\mathcal{M}$ is $\rho_p^\psi$-invariant, since $\rho_p^\psi(S)\left(\psi_p(x)k\right)=\psi_p(Sx)k\in \mathcal{M}$ for any $S\in \mathcal{L}_A\left(\mathbf{X}_p\right)$, $x\in \mathbf{X}_p$, and $k\in \mathcal{K}$. Now suppose that $\psi_e|_\mathcal{K}$ is faithful. Since $\mathcal{L}_A\left(\mathbf{X}_p\right)$ acts faithfully on $\mathbf{X}_p$,  the induced representation $\mathbf{X}_p\text{-}\mathrm{Ind}_A^{\mathcal{L}_A\left(\mathbf{X}_p\right)}\left(\psi_e|_\mathcal{K}\right):\mathcal{L}_A\left(\mathbf{X}_p\right)\rightarrow \mathcal{B}\left(\mathbf{X}_p\otimes_A \mathcal{K}\right)$ is faithful by \cite[Corollary~2.74]{MR1634408}. Since $U$ implements a unitary equivalence between $\mathbf{X}_p\text{-}\mathrm{Ind}_A^{\mathcal{L}_A\left(\mathbf{X}_p\right)}\left(\psi_e|_\mathcal{K}\right)$ and $\rho_p^\psi|_\mathcal{M}$, and unitary equivalence preserves the faithfulness of representations, we conclude that $\rho_p^\psi|_\mathcal{M}$ is faithful. 
\end{proof}
\end{prop}

We now show what a product of projections from the collection $\{P_p^\psi:p\in P\}$ looks like. 

\begin{prop}
Let $(G,P)$ be a quasi-lattice ordered group and $\mathbf{X}$ a compactly aligned product system over $P$ with coefficient algebra $A$. Let $\psi:\mathbf{X}\rightarrow \mathcal{B}(\mathcal{H})$ be a Nica covariant representation of $\mathbf{X}$ on a Hilbert space $\mathcal{H}$. Then for each $p,q\in P$, we have
\[
P_p^\psi P_q^\psi=P_{p\vee q}^\psi.
\]
In particular, the projections $\{P_p^\psi:p\in P\}$ commute.
\begin{proof}
Firstly, observe that part (i) of Proposition~\ref{existence of rho map} implies that $P_p^\psi =\rho_p^\psi\left(\mathrm{id}_{\mathbf{X}_p}\right)$ for any $p\in P\setminus \{e\}$. 

Next, we show that if $p\in P$ and $(e_i)_{i\in I}$ is the canonical approximate identity for the $C^*$-algebra $\mathcal{K}_A(\mathbf{X}_p)$, then
\begin{enumerate}[label=\upshape(\roman*)]
\item 
$\lim_{i\in I}(e_i x)=x$ for each $x\in \mathbf{X}_p$;
\item
$\lim_{i\in I}\big(\rho_p^\psi(e_i)\big)=\rho_p^\psi\left(\mathrm{id}_{\mathbf{X}_p}\right)$ (converging in the strong operator topology). 
\end{enumerate}
To see (i), fix $x\in \mathbf{X}_p$ and $\varepsilon>0$. 
Choose $x'\in \mathbf{X}_p$ so that $x=x'\cdot \langle x',x'\rangle_A$ by the Hewitt--Cohen--Blanchard factorisation theorem. Choose $i\in I$ such that for all $j\geq i$,
\[
\norm{e_j \Theta_{x',x'}-\Theta_{x',x'}}_{\mathcal{K}_A(\mathbf{X}_p)}<\frac{\varepsilon}{\norm{x'}_{\mathbf{X}_p}+1}.
\]
Thus, for all $j\geq i$, we have
\begin{align*}
\norm{e_j x-x}_{\mathbf{X}_p}
&=\norm{e_j x'\cdot \langle x',x'\rangle_A -x'\cdot \langle x',x'\rangle_A}_{\mathbf{X}_p}\\
&=\norm{\left(e_j \Theta_{x',x'}-\Theta_{x',x'}\right)x'}_{\mathbf{X}_p}
\leq \norm{e_j \Theta_{x',x'}-\Theta_{x',x'}}_{\mathcal{K}_A(\mathbf{X}_p)}\norm{x'}_{\mathbf{X}_p}
<\varepsilon.
\end{align*} 
Since $\varepsilon>0$ was arbitrary, we conclude that $\lim_{i\in I}(e_i x)=x$ for each $x\in \mathbf{X}_p$. 

We now move on to proving (ii). Fix $h\in \mathcal{H}$ and $\varepsilon>0$. If $h\in \left(\psi_p(\mathbf{X}_p)\mathcal{H}\right)^\perp$, then
\[
\rho_p^\psi(e_i)h=0=\rho_p^\psi\left(\mathrm{id}_{\mathbf{X}_p}\right)h
\]
for each $i\in I$. Thus, $\lim_{i\in I}\big\|\rho_p^\psi(e_i)h-\rho_p^\psi\left(\mathrm{id}_{\mathbf{X}_p}\right)h\big\|_{\mathcal{H}}=0$. On the other hand, if $h\in \overline{\psi_p(\mathbf{X}_p)\mathcal{H}}$, then we can choose $x_1,\ldots, x_n\in \mathbf{X}_p$ and $h_1,\ldots, h_n\in \mathcal{H}$ such that 
\[
\bigg\|h-\sum_{i=1}^n \psi_p(x_i)h_i\bigg\|_\mathcal{H}<\varepsilon/4.
\] 
Since $\norm{e_i}_{\mathcal{L}_A(\mathbf{X}_p)}\leq 1$ for each $i\in I$ and $\rho_p^\psi$ is norm-decreasing, we see that
\begin{align*}
\bigg\|\rho_p^\psi\left(e_i-\mathrm{id}_{\mathbf{X}_p}\right)\bigg(h-\sum_{i=1}^n \psi_p(x_i)h_i\bigg)\bigg\|_{\mathcal{H}}
&\leq \norm{\rho_p^\psi\left(e_i-\mathrm{id}_{\mathbf{X}_p}\right)}_{\mathcal{B}(\mathcal{H})}\bigg\|h-\sum_{i=1}^n \psi_p(x_i)h_i\bigg\|_{\mathcal{H}}\\
&\leq \norm{e_i-\mathrm{id}_{\mathbf{X}_p}}_{\mathcal{L}_A(\mathbf{X}_p)}\bigg\|h-\sum_{i=1}^n \psi_p(x_i)h_i\bigg\|_{\mathcal{H}}\\
&\leq 2\bigg\|h-\sum_{i=1}^n \psi_p(x_i)h_i\bigg\|_{\mathcal{H}}\\
&<\varepsilon/2.
\end{align*}
By (i), for each $1\leq i \leq n$, we can choose $j_i\in I$ such that whenever $k\geq j_i$,
\[
\norm{e_k x_i-x_i}_{\mathbf{X}_p}<\frac{\varepsilon}{2n\left(\max_{1\leq i\leq n}\norm{h_i}_\mathcal{H}+1\right)}.
\]
As $I$ is directed, we can choose $m\in I$ such that $m\geq j_i$ for each $1\leq i \leq n$. Since $\psi_p$ is norm-decreasing, we see that for any $k\geq m$, 
\begingroup
\allowdisplaybreaks
\begin{align*}
\bigg\|\rho_p^\psi\left(e_k-\mathrm{id}_{\mathbf{X}_p}\right)\bigg(\sum_{i=1}^n \psi_p(x_i)h_i\bigg)\bigg\|_{\mathcal{H}}
&=\bigg\|\sum_{i=1}^n \psi_p\left(\left(e_k-\mathrm{id}_{\mathbf{X}_p}\right)x\right)h_i\bigg\|_{\mathcal{H}}\\
&\leq \sum_{i=1}^n\norm{ \psi_p\left(\left(e_k-\mathrm{id}_{\mathbf{X}_p}\right)x\right)}_{\mathcal{B}(\mathcal{H})}\norm{h_i}_{\mathcal{H}}\\
&\leq \sum_{i=1}^n\norm{\left(e_k-\mathrm{id}_{\mathbf{X}_p}\right)x}_{\mathbf{X}_p}\norm{h_i}_{\mathcal{H}}\\
&< \sum_{i=1}^n \frac{\varepsilon\norm{h_i}_{\mathcal{H}}}{2n\left(\max_{1\leq i\leq n}\norm{h_i}_\mathcal{H}+1\right)}\\
&\leq \varepsilon/2.
\end{align*}
\endgroup
Thus, for each $k\geq m$, 
\begin{align*}
\big\|\rho_p^\psi&(e_k)h-\rho_p^\psi\left(\mathrm{id}_{\mathbf{X}_p}\right)h\big\|_{\mathcal{H}}\\
&=\big\|\rho_p^\psi\left(e_k-\mathrm{id}_{\mathbf{X}_p}\right)h\big\|_{\mathcal{H}}\\
&\leq \bigg\|\rho_p^\psi\left(e_k-\mathrm{id}_{\mathbf{X}_p}\right)\bigg(h-\sum_{i=1}^n \psi_p(x_i)h_i\bigg)\bigg\|_{\mathcal{H}}+
\bigg\|\rho_p^\psi\left(e_k-\mathrm{id}_{\mathbf{X}_p}\right)\bigg(\sum_{i=1}^n \psi_p(x_i)h_i\bigg)\bigg\|_{\mathcal{H}}\\
&<\varepsilon. 
\end{align*}
Since $\varepsilon>0$ was arbitrary, we conclude that $\lim_{i\in I}\big\|\rho_p^\psi(e_i)h-\rho_p^\psi\left(\mathrm{id}_{\mathbf{X}_p}\right)h\big\|_{\mathcal{H}}=0$ for each $h\in \mathcal{H}$. Thus, $\lim_{i\in I} \rho_p^\psi(e_i)=\rho_p^\psi\left(\mathrm{id}_{\mathbf{X}_p}\right)$ in the strong operator topology. 

Finally, we are ready to prove that $P_p^\psi P_q^\psi=P_{p\vee q}^\psi$ for every $p,q\in P$. Since $P_e^\psi=\mathrm{id}_\mathcal{H}$, the result is trivial when $p=e$ or $q=e$. Thus, we may as well suppose that $p,q\neq e$. Let $(e_i)_{i\in I}$ and $(f_j)_{j\in J}$ be the canonical approximate identities for $\mathcal{K}_A(\mathbf{X}_p)$ and $\mathcal{K}_A(\mathbf{X}_q)$ respectively. Then for any $i\in I$ and $j\in J$, Proposition~\ref{existence of rho map} and the Nica covariance of $\psi$ tell us that
\begin{align*}
\rho_p^\psi(e_i)\rho_q^\psi(f_j)
=\psi^{(p)}(e_i)\psi^{(q)}(f_j)
&=\begin{cases}
\psi^{(p\vee q)}\left(\iota_p^{p\vee q}(e_i)\iota_q^{p\vee q}(f_j)\right) & \text{if $p\vee q<\infty$}\\
0 & \text{otherwise}
\end{cases}\\
&=\begin{cases}
\rho_{p\vee q}^\psi\left(\iota_p^{p\vee q}(e_i)\iota_q^{p\vee q}(f_j)\right) & \text{if $p\vee q<\infty$}\\
0 & \text{otherwise}
\end{cases}\\
&=\begin{cases}
\rho_{p\vee q}^\psi\left(\iota_p^{p\vee q}(e_i)\right)\rho_{p\vee q}^\psi\left(\iota_q^{p\vee q}(f_j)\right) & \text{if $p\vee q<\infty$}\\
0 & \text{otherwise}
\end{cases}\\
&=\begin{cases}
\rho_{p}^\psi\left(e_i\right)P_{p\vee q}^\psi\rho_{q}^\psi\left(f_j\right)P_{p\vee q}^\psi & \text{if $p\vee q<\infty$}\\
0 & \text{otherwise.}
\end{cases}
\end{align*}
Hence, by (ii), we have 
\begin{align*}
P_p^\psi P_q^\psi
=\rho_p^\psi\left(\mathrm{id}_{\mathbf{X}_p}\right)\rho_q^\psi\left(\mathrm{id}_{\mathbf{X}_q}\right)
&=\begin{cases}
\rho_{p}^\psi\left(\mathrm{id}_{\mathbf{X}_p}\right)P_{p\vee q}^\psi\rho_{q}^\psi\left(\mathrm{id}_{\mathbf{X}_q}\right)P_{p\vee q}^\psi & \text{if $p\vee q<\infty$}\\
0 & \text{otherwise}
\end{cases}\\
&=\begin{cases}
\rho_{p\vee q}^\psi\left(\iota_p^{p\vee q}\left(\mathrm{id}_{\mathbf{X}_p}\right)\right)\rho_{p\vee q}^\psi\left(\iota_q^{p\vee q}\left(\mathrm{id}_{\mathbf{X}_q}\right)\right) & \text{if $p\vee q<\infty$}\\
0 & \text{otherwise}
\end{cases}\\
&=\begin{cases}
\rho_{p\vee q}^\psi\left(\iota_p^{p\vee q}\left(\mathrm{id}_{\mathbf{X}_p}\right)\iota_q^{p\vee q}\left(\mathrm{id}_{\mathbf{X}_q}\right)\right) & \text{if $p\vee q<\infty$}\\
0 & \text{otherwise}
\end{cases}\\
&=\begin{cases}
\rho_{p\vee q}^\psi\left(\mathrm{id}_{\mathbf{X}_{p\vee q}}\right) & \text{if $p\vee q<\infty$}\\
0 & \text{otherwise}
\end{cases}\\
&=\begin{cases}
P_{p\vee q}^\psi & \text{if $p\vee q<\infty$}\\
0 & \text{otherwise.}
\end{cases}
\end{align*}
Thus, $P_p^\psi P_q^\psi=P_{p\vee q}^\psi$ for each $p,q\in P$.
\end{proof}
\end{prop}

Our aim is to use the projections defined in Definition~\ref{basic projections} to construct a collection of mutually orthogonal projections in $\mathcal{B}(\mathcal{H})$ that decompose the identity and commute with everything in $\psi_*(\mathcal{NT}_\mathbf{X}^{\delta_\mathbf{X}})$. 

\begin{defn}
Let $(G,P)$ be a quasi-lattice ordered group and $\mathbf{X}$ a compactly aligned product system over $P$ with coefficient algebra $A$. Let $\psi:\mathbf{X}\rightarrow \mathcal{B}(\mathcal{H})$ be a Nica covariant representation of $\mathbf{X}$ on a Hilbert space $\mathcal{H}$. Let $F$ be a finite subset of $P$. For each $C\subseteq F$, define
\[
Q_{C,F}^\psi:=
P_{\bigvee C}^\psi\prod_{p\in F\setminus C}\left(\mathrm{id}_\mathcal{H}-P_p^\psi\right),
\]
where, by convention, the product over the empty set is $\mathrm{id}_\mathcal{H}$. 
\end{defn}

Whilst we have defined the projections $Q_{C,F}^\psi$, for every subset $C$ of $F$, we are particularly interested in the projections corresponding to so called initial segments of $F$. 

\begin{defn}
Let $(G,P)$ be a quasi-lattice ordered group. Let $F\subseteq P$ be a finite set. A subset $C\subseteq F$ is said to be an initial segment of $F$ if $\bigvee C<\infty$ and 
$C=\left\{t\in F:t\leq \bigvee C\right\}$.
\end{defn}

The next result shows how the projections $\{Q_{C,F}^\psi: \text{$C$ is an initial segment of $F$}\}$ and $\{P_p^\psi:p\in P\}$ interact with the operators $\{\psi_p(x):p\in P, \, x\in \mathbf{X}_p\}$. 

\begin{lem}
\label{projections and psi}
Let $(G,P)$ be a quasi-lattice ordered group and $\mathbf{X}$ a compactly aligned product system over $P$ with coefficient algebra $A$. Let $\psi:\mathbf{X}\rightarrow \mathcal{B}(\mathcal{H})$ be a Nica covariant representation of $\mathbf{X}$ on a Hilbert space $\mathcal{H}$. 
\begin{enumerate}[label=\upshape(\roman*)]
\item
Let $p,q\in P$ and $x\in \mathbf{X}_p$. Then
\[
P_q^\psi \psi_p(x)=
\begin{cases}
\psi_p(x)P_{p^{-1}(p\vee q)}^\psi & \text{if $p\vee q<\infty$}\\
0 &\text{otherwise.}
\end{cases}
\]
\item
If $F\subseteq P$ is finite and $p\in F$, then
\[
Q_{C,F}^\psi \psi_p(x)=
\begin{cases}
Q_{C,F}^\psi \psi_p(x)P_{p^{-1}\left(\bigvee C\right)}^\psi & \text{if $p\leq\bigvee C$}\\
0 &\text{otherwise}
\end{cases}
\]
for any initial segment $C$ of $F$. 
\end{enumerate}
\begin{proof}
Fix $p,q\in P$ and $x\in \mathbf{X}_p$. If $p\vee q=\infty$, then $\psi_q(y)^*\psi_p(x)=0$ for any $y\in \mathbf{X}_q$. Hence, for any $h,g\in \mathcal{H}$, it follows that
\begin{align*}
\langle \psi_q(y)h, \psi_p(x)g\rangle_\C =\langle h, \psi_q(y)^*\psi_p(x)g \rangle_\C=0.
\end{align*}
Thus, $\psi_p(\mathbf{X}_p)\mathcal{H}\subseteq \left(\psi_q(\mathbf{X}_q)\mathcal{H}\right)^\perp$, and so $P_q^\psi \psi_p(x)=0$.  

Now suppose that $p\vee q<\infty$. If $p=p\vee q$, then $p\geq q$, and so 
\[
P_q^\psi\psi_p(x)=\psi_p(x)=\psi_p(x)P_e^\psi=\psi_p(x)P_{p^{-1}(p\vee q)}^\psi.
\]
If $p\neq p\vee q$, then for any $y\in \mathbf{X}_{p^{-1}(p \vee q)}$ and $h\in \mathcal{H}$, we have
\begin{align*}
\psi_p(x)P_{p^{-1}(p \vee q)}^\psi\psi_{p^{-1}(p\vee q)}(y)h
&=\psi_p(x)\psi_{p^{-1}(p\vee q)}(y)h\\
&=\psi_{p\vee q}(xy)h
=P_q^\psi\psi_{p\vee q}(xy)h
=P_q^\psi\psi_p(x)\psi_{p^{-1}(p\vee q)}(y)h.
\end{align*}
Consequently, $\psi_p(x)P_{p^{-1}(p \vee q)}^\psi$ and $P_q^\psi\psi_p(x)$ agree on $\overline{\psi_{p^{-1}(p\vee q)}(\mathbf{X}_{p^{-1}(p\vee q)})\mathcal{H}}$. It remains to check that they agree on the orthogonal complement $\left(\psi_{p^{-1}(p\vee q)}(\mathbf{X}_{p^{-1}(p\vee q)})\mathcal{H}\right)^\perp$. Let $f\in \left(\psi_{p^{-1}(p\vee q)}(\mathbf{X}_{p^{-1}(p\vee q)})\mathcal{H}\right)^\perp$. We need to show that $P_q^\psi\psi_p(x)f=0$. It suffices to show that $\psi_p(x)f\in \left(\psi_q(\mathbf{X}_q)\mathcal{H}\right)^\perp$: for any $y\in \mathbf{X}_q$ and $h\in \mathcal{H}$, we have
\begin{align*}
\langle \psi_p(x)f,\psi_q(y)h \rangle_\C
&=\langle f, \psi_p(x)^*\psi_q(y)h \rangle_\C\\
&\in \left\langle f, \cspan\{\psi_{p^{-1}(p\vee q)}(\mathbf{X}_{p^{-1}(p\vee q)})\psi_{q^{-1}(p\vee q)}(\mathbf{X}_{q^{-1}(p\vee q)})^*\}\mathcal{H}\right\rangle_\C\\
&\subseteq \left\langle f, \psi_{p^{-1}(p\vee q)}(\mathbf{X}_{p^{-1}(p\vee q)})\mathcal{H}\right\rangle_\C\\
&=\{0\}.
\end{align*}
This completes the proof of (i).

We now prove part (ii). Fix a finite set $F\subseteq P$ with $p\in F$. Let $C$ be an initial segment of $F$. If $p\leq \bigvee C$, then $p \vee \left(\bigvee C\right)=\bigvee C<\infty$. By part (i) it follows that
\[
Q_{C,F}^\psi \psi_p(x)=Q_{C,F}^\psi P_{\bigvee C}^\psi \psi_p(x)=Q_{C,F}^\psi \psi_p(x)P_{p^{-1}\left(p\vee \left(\bigvee C\right)\right)}^\psi=Q_{C,F}^\psi \psi_p(x)P_{p^{-1} \left(\bigvee C\right)}^\psi.
\]
On the other hand, suppose that $p\not \leq \bigvee C$. Since $p\in F$, this implies that $C\neq F$. Moreover, since $C$ is an initial segment of $F$, this forces $p\in F\setminus C$. Therefore,
\[
Q_{C,F}^\psi \psi_p(x)=Q_{C,F}^\psi \big(\mathrm{id}_\mathcal{H}-P_p^\psi\big) \psi_p(x)=0. 
\]
This completes the proof of part (ii). 
\end{proof}
\end{lem}

Using the previous result we can show that every element of $\psi_*(\mathcal{NT}_\mathbf{X}^{\delta_\mathbf{X}})$ commutes with the projections $\{P_p^\psi:p\in P\}$ and $\{Q_{C,F}^\psi: C\subseteq F\}$. 

\begin{prop}
\label{projections commute with core}
For any $q\in P$, the projection $P_q^\psi$ commutes with every element of $\psi_*(\mathcal{NT}_\mathbf{X}^{\delta_\mathbf{X}})$. In particular, if $F\subseteq P$ is finite and $C\subseteq F$, then $Q_{C,F}^\psi$ commutes with every element of $\psi_*(\mathcal{NT}_\mathbf{X}^{\delta_\mathbf{X}})$.
\begin{proof}
Since $\psi_*(\mathcal{NT}_\mathbf{X}^{\delta_\mathbf{X}})=\cspan \{\psi_p(\mathbf{X}_p)\psi_p(\mathbf{X}_p)^*\}$, it suffices to show that 
\[
P_q^\psi\psi_p(x)\psi_p(y)^*=\psi_p(x)\psi_p(y)^*P_q^\psi
\]
for each $x,y\in \mathbf{X}_p$. Via two applications of Lemma~\ref{projections and psi}, we see that
\begin{align*}
P_q^\psi\psi_p(x)\psi_p(y)^*
&=
\begin{cases}
\psi_p(x)P_{p^{-1}(p\vee q)}^\psi\psi_p(y)^* & \text{if $p\vee q<\infty$}\\
0 &\text{otherwise}
\end{cases}\\
&=
\begin{cases}
\psi_p(x)\big(\psi_p(y)P_{p^{-1}(p\vee q)}^\psi\big)^* & \text{if $p\vee q<\infty$}\\
0 &\text{otherwise}
\end{cases}\\
&=\psi_p(x)\big(P_q^\psi\psi_p(y)\big)^*\\
&=\psi_p(x)\psi_p(y)^*P_q^\psi,
\end{align*}
as required. 
\end{proof}
\end{prop}

We now show that the projections $\{Q_{C,F}^\psi: \text{$C$ is an initial segment of $F$}\}$ are mutually orthogonal and decompose the identity operator on $\mathcal{H}$.

\begin{prop}
Let $(G,P)$ be a quasi-lattice ordered group and $\mathbf{X}$ a compactly aligned product system over $P$ with coefficient algebra $A$. Let $\psi:\mathbf{X}\rightarrow \mathcal{B}(\mathcal{H})$ be a Nica covariant representation of $\mathbf{X}$. Let $F\subseteq P$ be finite. Then
\begin{enumerate}[label=\upshape(\roman*)]
\item if $C\subseteq F$ is not an initial segment of $F$, then $Q_{C,F}^\psi=0$;
\item $\{Q_{C,F}^\psi:C\subseteq F \text{ is an initial segment of } F\}$ is a decomposition of the identity on $\mathcal{H}$ into mutually orthogonal projections. 
\end{enumerate}
\begin{proof}
Suppose $C\subseteq F$ is not an initial segment of $F$. If $\bigvee C=\infty$, then 
\[
Q_{C,F}^\psi=Q_{C,F}^\psi P_{\bigvee C}^\psi=Q_{C,F}^\psi P_\infty^\psi=0.
\]
Alternatively, $\bigvee C<\infty$ and $C\neq \{t\in F:t\leq \bigvee C\}$. Thus, $C\neq F$. Choose $t\in F\setminus C$ with $t\leq \bigvee C$. Since $t \vee  \left(\bigvee C\right)=\bigvee C$, we see that
\begin{align*}
Q_{C,F}^\psi
=Q_{C,F}^\psi P_{\bigvee C}^\psi(\mathrm{id}_\mathcal{H}-P_t^\psi)
&=Q_{C,F}^\psi\big(P_{\bigvee C}^\psi- P_{\bigvee C}^\psi P_t^\psi\big)\\
&=Q_{C,F}^\psi\big(P_{\bigvee C}^\psi-P_{t\vee  \left(\bigvee C\right)}^\psi\big)
=Q_{C,F}^\psi\big(P_{\bigvee C}^\psi-P_{\bigvee C}^\psi\big)
=0.
\end{align*}
Thus, $Q_{C,F}^\psi=0$, which proves part (i). 

We now prove part (ii). Since $Q_{C,F}^\psi=0$ whenever $C$ is not an initial segment of $F$, it suffices to show that $\{Q_{C,F}^\psi:C\subseteq F\}$ is a decomposition of the identity into mutually orthogonal projections. Firstly, we show orthogonality. Suppose $C,D\subseteq F$ are distinct. Without loss of generality, we may assume that $D\setminus C\neq \emptyset$. Thus, $C\neq F$ and we can choose $t\in D\setminus C$.  Since $t \vee  \left(\bigvee D\right)=\bigvee D$, we have
\begin{align*}
Q_{C,F}^\psi Q_{D,F}^\psi
=Q_{C,F}^\psi P_{\bigvee C}^\psi \big(\mathrm{id}_\mathcal{H} -P_t^\psi\big)P_{\bigvee D}^\psi Q_{D,F}^\psi
&=Q_{C,F}^\psi \big(P_{\bigvee C}^\psi -P_{t\vee  \left(\bigvee C\right)}^\psi\big)P_{\bigvee D}^\psi Q_{D,F}^\psi\\
&=Q_{C,F}^\psi \big(P_{\left(\bigvee C\right)\vee \left(\bigvee D\right)}^\psi -P_{t\vee  \left(\bigvee C\right) \vee  \left(\bigvee D\right)}^\psi\big) Q_{D,F}^\psi\\
&=Q_{C,F}^\psi \big(P_{\left(\bigvee C\right)\vee \left(\bigvee D\right)}^\psi -P_{\left(\bigvee C\right) \vee \left(\bigvee D\right)}^\psi\big) Q_{D,F}^\psi\\
&=0. 
\end{align*} 
It remains to check that $\sum_{C\subseteq F}Q_{C,F}^\psi=\mathrm{id}_\mathcal{H}$. To prove this, we will use induction on $|F|$. When $|F|=0$ we have
\begin{align*}
\sum_{C\subseteq F}Q_{C,F}^\psi=Q_{\emptyset,\emptyset}^\psi=P_{\bigvee \emptyset}^\psi=P_e^\psi=\mathrm{id}_\mathcal{H}.
\end{align*}
Now let $n\geq 0$ and suppose that $\sum_{C\subseteq F}Q_{C,F}^\psi=\mathrm{id}_\mathcal{H}$ whenever $F\subseteq P$ and $|F|=n$. Fix $F'\subseteq P$ with $|F'|=n+1$. Then, for any $y\in F'$, we have
\begin{align*}
\sum_{C\subseteq F'}Q_{C,F'}^\psi
&=\sum_{C\subseteq F', \, y\in C}Q_{C,F'}^\psi+\sum_{C\subseteq F', \, y\not\in C}Q_{C,F'}^\psi\\
&=\sum_{C\subseteq F', \, y\in C}P_{\bigvee C}^\psi\prod_{p\in F'\setminus C}\big(\mathrm{id}_\mathcal{H}-P_p^\psi\big)+\sum_{C\subseteq F', \, y\not\in C}P_{\bigvee C}^\psi\prod_{p\in F'\setminus C}\big(\mathrm{id}_\mathcal{H}-P_p^\psi\big)\\
&=\sum_{C\subseteq F'\setminus \{y\}}P_{\bigvee (C\cup \{y\})}^\psi\prod_{p\in (F'\setminus \{y\})\setminus C}\big(\mathrm{id}_\mathcal{H}-P_p^\psi\big)\\
& \quad \quad \quad \quad \quad \quad 
+\sum_{C\subseteq F'\setminus \{y\}}P_{\bigvee C}^\psi\prod_{p\in (F'\setminus \{y\})\setminus C}\big(\mathrm{id}_\mathcal{H}-P_p^\psi\big)\big(\mathrm{id}_\mathcal{H}-P_y^\psi\big)\\
&=\sum_{C\subseteq F'\setminus \{y\}}P_{\bigvee C}^\psi P_y^\psi\prod_{p\in (F'\setminus \{y\})\setminus C}\big(\mathrm{id}_\mathcal{H}-P_p^\psi\big)\\
& \quad \quad \quad \quad \quad \quad 
+\sum_{C\subseteq F'\setminus \{y\}}P_{\bigvee C}^\psi\prod_{p\in (F'\setminus \{y\})\setminus C}\big(\mathrm{id}_\mathcal{H}-P_p^\psi\big)\big(\mathrm{id}_\mathcal{H}-P_y^\psi\big)\\
&=\big(P_y^\psi+\big(\mathrm{id}_\mathcal{H}-P_y^\psi\big)\big)\sum_{C\subseteq F'\setminus \{y\}}P_{\bigvee C}^\psi \prod_{p\in (F'\setminus \{y\})\setminus C}\big(\mathrm{id}_\mathcal{H}-P_p^\psi\big)\\
&=\sum_{C\subseteq F'\setminus \{y\}}P_{\bigvee C}^\psi \prod_{p\in (F'\setminus \{y\})\setminus C}\big(\mathrm{id}_\mathcal{H}-P_p^\psi\big)\\
&=\sum_{C\subseteq F'\setminus \{y\}}Q_{C,F'\setminus \{y\}}^\psi\\
&=\mathrm{id}_\mathcal{H},
\end{align*}
where the last equality follows from applying the inductive hypothesis to $F'\setminus \{y\}$. 
\end{proof}
\end{prop}

Putting these results together we get an expression for the norm of an element in $\psi_*(\mathcal{NT}_\mathbf{X}^{\delta_\mathbf{X}})$ that does not depend on the representation $\psi$. 

\begin{lem}
\label{norm of things in core}
Let $(G,P)$ be a quasi-lattice ordered group and $\mathbf{X}$ a compactly aligned product system over $P$ with coefficient algebra $A$. Let $\psi:\mathbf{X}\rightarrow \mathcal{B}(\mathcal{H})$ be a Nica covariant representation of $\mathbf{X}$ on a Hilbert space $\mathcal{H}$. If $Z:=\sum_k \psi_{p_k}(x_k)\psi_{p_k}(y_k)^*\in \psi_*(\mathcal{NT}_\mathbf{X}^{\delta_\mathbf{X}})$ is a finite sum, then
\[
\norm{Z}_{\mathcal{B}(\mathcal{H})}=\max\bigg\{\bigg\|Q_{C,F}^\psi \rho_{\bigvee C}^\psi\bigg(\sum_k \iota_{p_k}^{\bigvee C}\left(\Theta_{x_k,y_k}\right)\bigg)\bigg\|_{\mathcal{B}(\mathcal{H})}: \text{$C$ is an initial segment of $F$}\bigg\}
\]
for any finite set $F\subseteq P$ containing each $p_k$. Furthermore, if the representation
\begin{align*}
A\ni a \mapsto \psi_e(a)\prod_{t\in K}\big(\mathrm{id}_\mathcal{H}-P_t^\psi\big)\in \mathcal{B}(\mathcal{H})
\end{align*}
is faithful for any finite set $K\subseteq P\setminus \{e\}$, then
\[
\norm{Z}_{\mathcal{B}(\mathcal{H})}=\max\bigg\{\bigg\|\sum_k \iota_{p_k}^{\bigvee C}\left(\Theta_{x_k,y_k}\right)\bigg\|_{\mathcal{L}_A\left(\mathbf{X}_{\bigvee C}\right)}: \text{$C$ is an initial segment of $F$}\bigg\}.
\]
\begin{proof}
Let $F$ be a finite subset of $P$ containing each $p_k$. Since $Q_{C,F}^\psi$ commutes with $Z$ for each $C\subseteq F$ (by Proposition~\ref{projections commute with core}) and $\{Q_{C,F}^\psi: \text{$C$ is an initial segment of $F$}\}$ is an orthogonal decomposition of the identity, we have that
\[
\norm{Z}_{\mathcal{B}(\mathcal{H})}=\max\left\{\norm{Q_{C,F}^\psi Z}_{\mathcal{B}(\mathcal{H})}: \text{$C$ is an initial segment of $F$}\right\}.
\]
However, for any initial segment $C$ of $F$, Lemma~\ref{projections and psi} shows that
\begin{align*}
Q_{C,F}^\psi Z
=Q_{C,F}^\psi  \sum_k \psi_{p_k}(x_k)\psi_{p_k}(y_k)^*
=Q_{C,F}^\psi  \sum_{k:p_k\leq \bigvee C} \psi_{p_k}(x_k)\psi_{p_k}(y_k)^*P_{\bigvee C}^\psi
\end{align*}
Using parts (ii) and (iii) of Proposition~\ref{existence of rho map}, this is equal to
\begin{align*}
Q_{C,F}^\psi  \sum_{k:p_k\leq \bigvee C} \rho_{p_k}^\psi\left(\Theta_{x_k,y_k}\right)P_{\bigvee C}^\psi
&=Q_{C,F}^\psi  \sum_{k:p_k\leq \bigvee C} \rho_{\bigvee C}^\psi\big(\iota_{p_k}^{\bigvee C}\left(\Theta_{x_k,y_k}\right)\big)\\
&=Q_{C,F}^\psi  \rho_{\bigvee C}^\psi\bigg(\sum_k\iota_{p_k}^{\bigvee C}\left(\Theta_{x_k,y_k}\right)\bigg).
\end{align*}
Now suppose that for any finite set $K\subseteq P\setminus \{e\}$, the representation
\begin{align*}
A\ni a \mapsto \psi_e(a)\prod_{t\in K}\big(\mathrm{id}_\mathcal{H}-P_t^\psi\big)\in \mathcal{B}(\mathcal{H})
\end{align*}
is faithful. To complete the proof we will show that the representation
\[
\mathcal{L}_A\left(\mathbf{X}_{\bigvee C}\right)\ni T \mapsto Q_{C,F}^\psi  \rho_{\bigvee C}^\psi(T)\in \mathcal{B}(\mathcal{H})
\]
is faithful, and hence
\begin{align*}
\norm{Z}_{\mathcal{B}(\mathcal{H})}&=\max\left\{\norm{Q_{C,F}^\psi Z}_{\mathcal{B}(\mathcal{H})}: \text{$C$ is an initial segment of $F$}\right\}\\
&=\max\bigg\{\bigg\|Q_{C,F}^\psi  \rho_{\bigvee C}^\psi\bigg(\sum_k\iota_{p_k}^{\bigvee C}\left(\Theta_{x_k,y_k}\right)\bigg)\bigg\|_{\mathcal{B}(\mathcal{H})}: \text{$C$ is an initial segment of $F$}\bigg\}\\
&=\max\bigg\{\bigg\|\sum_k\iota_{p_k}^{\bigvee C}\left(\Theta_{x_k,y_k}\right)\bigg\|_{\mathcal{L}_A\left(\mathbf{X}_{\bigvee C}\right)}: \text{$C$ is an initial segment of $F$}\bigg\}.
\end{align*}
Let $\mathcal{K}:=\prod_{\left\{t\in F\setminus C:t\vee \left(\bigvee C\right) <\infty\right\}}\left(\mathrm{id}_\mathcal{H}-P_{\left(\bigvee C\right)^{-1}\left(t\vee \left(\bigvee C\right)\right)}^\psi\right)\mathcal{H}$. Since $\psi_e(a)P_p^\psi=P_p^\psi \psi_e(a)$, for each $a\in A$ and $p\in P$, by Lemma~\ref{projections and psi}, we see that $\mathcal{K}$ is a $\psi_e$-invariant subspace of $\mathcal{H}$. As $C$ is an initial segment of $F$, if $t\in F\setminus C$ with $t\vee \left(\bigvee C\right) <\infty$, then $t\not\leq \bigvee C$, and so $\left(\bigvee C\right)^{-1}\left(t\vee\left(\bigvee C\right)\right)\neq e$. Thus, $\psi_e |_\mathcal{K}$ is faithful. Therefore, by Proposition~\ref{existence of rho map} it follows that $\mathcal{M}:=\overline{\psi_{\bigvee C}(\mathbf{X}_{\bigvee C})\mathcal{K}}$ is a $\rho_{\bigvee C}^\psi$-invariant subspace and $\rho_{\bigvee C}^\psi|_\mathcal{M}$ is faithful. To show that the map $\mathcal{L}_A\left(\mathbf{X}_{\bigvee C}\right)\ni T \mapsto Q_{C,F}^\psi  \rho_{\bigvee C}^\psi(T)\in \mathcal{B}(\mathcal{H})$ is faithful, it remains to check that $\mathcal{M}\subseteq Q_{C,F}^\psi\mathcal{H}$. Lemma~\ref{projections and psi} tells us that for any $x\in \mathbf{X}_{\bigvee C}$, we have 
\begin{align*}
Q_{C,F}^\psi\psi_{\bigvee C}(x)
=P_{\bigvee C}^\psi\hspace{-0.2em}\prod_{t\in F\setminus C}\big(\mathrm{id}_\mathcal{H}-P_t^\psi\big)\psi_{\bigvee C}(x)
=\psi_{\bigvee C}(x)\hspace{-0.7em}\prod_{\substack{t\in F\setminus C,\\ t\vee \hspace{-1em}\left(\bigvee C\right)<\infty}}\left(\mathrm{id}_\mathcal{H}-P_{\left(\bigvee C\right)^{-1}\left(t\vee \left(\bigvee C\right)\right)}^\psi\right).
\end{align*}
Therefore, $Q_{C,F}^\psi$ is the identity on 
\[
\overline{\psi_{\bigvee C}(\mathbf{X}_{\bigvee C})\prod_{\substack{t\in F\setminus C,\\ t\vee \left(\bigvee C\right) <\infty}}\left(\mathrm{id}_\mathcal{H}-P_{\left(\bigvee C\right)^{-1}\left(t\vee \left(\bigvee C\right)\right)}^\psi\right)\mathcal{H}}
=\overline{\psi_{\bigvee C}\left(\mathbf{X}_{\bigvee C}\right)\mathcal{K}}
=\mathcal{M},
\] and so $\mathcal{M}\subseteq Q_{C,F}^\psi\mathcal{H}$.
\end{proof}
\end{lem}

Now that we have an expression for the norm of elements in $\psi_*(\mathcal{NT}_\mathbf{X}^{\delta_\mathbf{X}})$, we are ready to show that the expectation $E_{\delta_\mathbf{X}}$ can be implemented spatially.

\begin{prop}
Let $(G,P)$ be a quasi-lattice ordered group and $\mathbf{X}$ a compactly aligned product system over $P$ with coefficient algebra $A$. Let $\psi:\mathbf{X}\rightarrow \mathcal{B}(\mathcal{H})$ be a Nica covariant representation of $\mathbf{X}$ on a Hilbert space $\mathcal{H}$. Suppose that for any finite set $K\subseteq P\setminus \{e\}$, the representation
\begin{align*}
A\ni a \mapsto \psi_e(a)\prod_{t\in K}\big(\mathrm{id}_\mathcal{H}-P_t^\psi\big)\in \mathcal{B}(\mathcal{H})
\end{align*}
is faithful. Then 
\begin{enumerate}[label=\upshape(\roman*)]
\item
$\psi_*|_{\mathcal{NT}_\mathbf{X}^{\delta_\mathbf{X}}}$ is faithful; and
\item
there exists a linear map $E_\psi:\psi_*\left(\mathcal{NT}_\mathbf{X}\right)\rightarrow \psi_*(\mathcal{NT}_\mathbf{X}^{\delta_\mathbf{X}})$ such that 
\[
E_\psi \circ \psi_*= \psi_* \circ E_{\delta_\mathbf{X}}.
\]
\end{enumerate}
\begin{proof}
Firstly, we prove that the restriction of $\psi_*$ to $\mathcal{NT}_\mathbf{X}^{\delta_\mathbf{X}}$ is faithful. Fix a finite sum $Z:=\sum_k i_{\mathbf{X}_{p_k}}(x_k)i_{\mathbf{X}_{p_k}}(y_k)^*\in \mathcal{NT}_\mathbf{X}^{\delta_\mathbf{X}}$. Let $\sigma:\mathcal{NT}_\mathbf{X}\rightarrow \mathcal{B}(\mathcal{H}')$ be a faithful representation. Thus, $\sigma\circ i_\mathbf{X}$ is a Nica covariant representation of $\mathbf{X}$ on $\mathcal{H}'$. For any finite set $F\subseteq P$ containing each $p_k$, two applications of Lemma~\ref{norm of things in core} show that 
\begingroup
\allowdisplaybreaks
\begin{align*}
\|&Z\|_{\mathcal{NT}_\mathbf{X}}
=
\norm{\sigma(Z)}_{\mathcal{B}(\mathcal{H}')}
=
\bigg\|\sum_k (\sigma\circ i_\mathbf{X})_{p_k}(x_k)(\sigma\circ i_\mathbf{X})_{p_k}(y_k)^*\bigg\|_{\mathcal{B}(\mathcal{H}')}\\
&=
\max\bigg\{\bigg\|Q_{C,F}^{\sigma \circ i_\mathbf{X}} \rho_{\bigvee C}^{\sigma \circ i_\mathbf{X}}\bigg(\sum_k \iota_{p_k}^{\bigvee C}\left(\Theta_{x_k,y_k}\right)\bigg)\bigg\|_{\mathcal{B}(\mathcal{H}')}:\text{$C$ is an initial segment of $F$}\bigg\}\\
&\leq
\max\bigg\{\bigg\|\sum_k \iota_{p_k}^{\bigvee C}\left(\Theta_{x_k,y_k}\right)\bigg\|_{\mathcal{L}_A\left(\mathbf{X}_{\bigvee C}\right)}: \text{$C$ is an initial segment of $F$}\bigg\}\\
&=\bigg\|\sum_k \psi_{p_k}(x_k)\psi_{p_k}(y_k)^*\bigg\|_{\mathcal{B}(\mathcal{H})}\\
&=\norm{\psi_*(Z)}_{\mathcal{B}(\mathcal{H})},
\end{align*}
\endgroup
where we used the fact that each $Q_{C,F}^{\sigma \circ i_\mathbf{X}}$ is a projection and each $*$-homomorphism $\rho_{\bigvee C}^{\sigma \circ i_\mathbf{X}}$ is norm-decreasing.
Thus, $\psi_*|_{\mathcal{NT}_\mathbf{X}^{\delta_\mathbf{X}}}$ is faithful. 

Next we prove part (ii). We first show that for any finite sum $\sum_k \psi_{p_k}(x_k)\psi_{q_k}(y_k)^*$, we have
\begin{equation}
\label{norm estimate onto core}
\begin{aligned}
\bigg\|\sum_{k:p_k=q_k}\psi_{p_k}(x_k)\psi_{p_k}(y_k)^*\bigg\|_{\mathcal{B}(\mathcal{H})}\leq \bigg\|\sum_k \psi_{p_k}(x_k)\psi_{q_k}(y_k)^*\bigg\|_{\mathcal{B}(\mathcal{H})}. 
\end{aligned}
\end{equation}
Let $F\subseteq P$ be the finite set consisting of each $p_k$ and $q_k$. By Lemma~\ref{norm of things in core}, there exists an initial segment $C$ of $F$ such that 
\[
\bigg\|\sum_{k:p_k=q_k}\psi_{p_k}(x_k)\psi_{p_k}(y_k)^*\bigg\|_{\mathcal{B}(\mathcal{H})}=\bigg\|\sum_{k:p_k=q_k}\iota_{p_k}^{\bigvee C}(\Theta_{x_k,y_k})\bigg\|_{\mathcal{L}_A\left(\mathbf{X}_{\bigvee C}\right)}.
\] 
For each $s,t\in C$ with $s\neq t$ and $\left(s^{-1}\left(\bigvee C\right)\right)\vee \left(t^{-1}\left(\bigvee C\right)\right)<\infty$ define
\begin{align*}
\beta_{s,t}
:=
\begin{cases}
s\left(\left(s^{-1}\left(\bigvee C\right)\right)\vee \left(t^{-1}\left(\bigvee C\right)\right)\right) &\text{if $s^{-1}\left(\bigvee C\right)<\left(s^{-1}\left(\bigvee C\right)\right)\vee \left(t^{-1}\left(\bigvee C\right)\right)$}\\
t\left(\left(s^{-1}\left(\bigvee C\right)\right)\vee \left(t^{-1}\left(\bigvee C\right)\right)\right) &\text{otherwise.}
\end{cases}
\end{align*}
Observe that for any $s,t\in C$ with $s\neq t$ and $\left(s^{-1}\left(\bigvee C\right)\right)\vee \left(t^{-1}\left(\bigvee C\right)\right)<\infty$ we have 
\[
s\big(\big(s^{-1}\big(\bigvee C\big)\big)\vee \big(t^{-1}\big(\bigvee C\big)\big)\big)\geq s\big(s^{-1}\big(\bigvee C\big)\big)=\bigvee C
\]
and
\[
t\big(\big((s^{-1}\big(\bigvee C\big)\big)\vee \big(t^{-1}\big(\bigvee C\big)\big)\big)\geq t\big(t^{-1}\big(\bigvee C\big)\big)=\bigvee C.
\] 
Thus, $\beta_{s,t}\geq \bigvee C$. Hence, $P_{\bigvee C}^\psi P^\psi_{\beta_{s,t}}=P^\psi_{\left(\bigvee C\right) \vee \beta_{s,t}}=P^\psi_{\beta_{s,t}}$, and we can define a projection
\[
R_{C,F}^\psi:=Q_{C,F}^\psi\hspace{-0.7em}\prod_{\substack{s,t\in C, \, s\neq t, \\ \left(s^{-1}\left(\bigvee C\right)\right)\vee \left(t^{-1}\left(\bigvee C\right)\right)<\infty}} \hspace{-1em}\big(P^\psi_{\bigvee C}-P^\psi_{\beta_{s,t}}\big).
\]
We claim that 
\[
R_{C,F}^\psi\bigg(\sum_k \psi_{p_k}(x_k)\psi_{q_k}(y_k)^*\bigg)R_{C,F}^\psi=R_{C,F}^\psi\sum_{k:p_k=q_k}\psi_{p_k}(x_k)\psi_{p_k}(y_k)^*.
\]
Since $R_{C,F}^\psi$ commutes with $\sum_{k:p_k=q_k}\psi_{p_k}(x_k)\psi_{p_k}(y_k)^*$ by Proposition~\ref{projections commute with core}, it suffices to show that $R_{C,F}^\psi\psi_p(x)\psi_q(y)^*R_{C,F}^\psi$=0 whenever $x\in \mathbf{X}_p$, $y\in \mathbf{X}_q$, with $p,q\in F$ and $p\neq q$. Firstly, if $p\not \in C$ or $q\not \in C$, then $p\not \leq \bigvee C$  or $q\not \leq \bigvee C$ (since $C$ is an initial segment of $F$), and so by Lemma~\ref{projections and psi} we have $Q_{C,F}^\psi\psi_p(x)=0$ or $Q_{C,F}^\psi\psi_q(y)=0$. Consequently, 
\begin{align*}
R_{C,F}^\psi\psi_p(x)\psi_q(y)^*R_{C,F}^\psi
&=R_{C,F}^\psi Q_{C,F}^\psi\psi_p(x)\psi_q(y)^*Q_{C,F}^\psi R_{C,F}^\psi\\
&=R_{C,F}^\psi Q_{C,F}^\psi\psi_p(x)(Q_{C,F}^\psi\psi_q(y))^*R_{C,F}^\psi\\
&=0.
\end{align*}
Alternatively, if $p,q\leq \bigvee C$ and $\left(p^{-1}\left(\bigvee C\right)\right)\vee \left(q^{-1}\left(\bigvee C\right)\right)=\infty$, then 
\[
P^\psi_{p^{-1}\left(\bigvee C\right)}P^\psi_{q^{-1}\left(\bigvee C\right)} =0.
\] 
Hence, Lemma~\ref{projections and psi} tells us that 
\begin{align*}
R_{C,F}^\psi\psi_p(x)\psi_q(y)^*R_{C,F}^\psi
&=R_{C,F}^\psi Q_{C,F}^\psi\psi_p(x)\psi_q(y)^*Q_{C,F}^\psi R_{C,F}^\psi\\
&=R_{C,F}^\psi Q_{C,F}^\psi\psi_p(x)P^\psi_{p^{-1}\left(\bigvee C\right)}P^\psi_{q^{-1}\left(\bigvee C\right)} \psi_q(y)^*Q_{C,F}^\psi R_{C,F}^\psi
=0.
\end{align*}
With this in mind, suppose that $p,q\in C$ and $\left(p^{-1}\left(\bigvee C\right)\right)\vee \left(q^{-1}\left(\bigvee C\right)\right)<\infty$. Since $p$ and $q$ are distinct, it follows that either $p^{-1}\left(\bigvee C\right)<\left(p^{-1}\left(\bigvee C\right)\right)\vee \left(q^{-1}\left(\bigvee C\right)\right)$ or $q^{-1}\left(\bigvee C\right)<\left(p^{-1}\left(\bigvee C\right)\right)\vee \left(q^{-1}\left(\bigvee C\right)\right)$. By taking adjoints, we may assume, without loss of generality, that $p^{-1}<\left(p^{-1}\left(\bigvee C\right)\right)\vee \left(q^{-1}\left(\bigvee C\right)\right)$. Therefore, 
\[
\beta_{p,q}=p\big(\big(p^{-1}\big(\bigvee C\big)\big)\vee \big(q^{-1}\big(\bigvee C\big)\big)\big).
\]
Consequently, an application of Lemma~\ref{projections and psi} shows that
\begin{align*}
R_{C,F}^\psi&\psi_p(x)\psi_q(y)^*R_{C,F}^\psi\\
&=R_{C,F}^\psi \big(P^\psi_{\bigvee C}-P^\psi_{\beta_{p,q}}\big)\psi_p(x)\psi_q(y)^* P^\psi_{\bigvee C}R_{C,F}^\psi\\
&=R_{C,F}^\psi\psi_p(x)\big(P^\psi_{p^{-1}\left(\bigvee C\right)}-P^\psi_{p^{-1}\beta_{p,q}}\big)P^\psi_{q^{-1}\left(\bigvee C\right)}\psi_q(y)^*R_{C,F}^\psi\\
&=R_{C,F}^\psi\psi_p(x)\big(P^\psi_{p^{-1}\left(\bigvee C\right)}-P^\psi_{\left(p^{-1}\left(\bigvee C\right)\right)\vee \left(q^{-1}\left(\bigvee C\right)\right)}\big)P^\psi_{q^{-1}\left(\bigvee C\right)}\psi_q(y)^*R_{C,F}^\psi\\
&=0. 
\end{align*}

Additionally, we claim that the representation
\[
\mathcal{L}_A\left(\mathbf{X}_{\bigvee C}\right)\ni T \mapsto R_{C,F}^\psi\rho_{\bigvee C}^\psi(T)\in \mathcal{B}(\mathcal{H})
\]
is faithful. Let 
\[
\mathcal{K}:=\hspace{-1em}\prod_{\substack{p\in F\setminus C, \,  s,t\in C,\\ p\vee \left(\bigvee C\right)<\infty, \, s\neq t,\\ \left(s^{-1}\left(\bigvee C\right)\right)\vee \left(t^{-1}\left(\bigvee C\right)\right)<\infty}}\hspace{-1em}\left(\mathrm{id}_\mathcal{H}-P_{\left(\bigvee C\right)^{-1}\left(p\vee \left(\bigvee C\right)\right)}^\psi\right)
\left(\mathrm{id}_\mathcal{H}-P^\psi_{\left(\bigvee C\right)^{-1}\beta_{s,t}}\right)\mathcal{H}.
\]
As $\psi_e(a)P_q^\psi=P_q^\psi \psi_e(a)$, for each $a\in A$ and $q\in P$, by Lemma~\ref{projections and psi}, we see that $\mathcal{K}$ is a $\psi_e$-invariant subspace of $\mathcal{H}$. Since $C$ is an initial segment of $F$, if $p\in F\setminus C$ with $p\vee\left(\bigvee C\right) <\infty$, then $p\not\leq \bigvee C$, and so $\left(\bigvee C\right)^{-1}\left(p\vee \left(\bigvee C\right)\right)\neq e$. We also claim that for any $s,t\in C$ with $s\neq t$ and $\left(s^{-1}\left(\bigvee C\right)\right)\vee \left(t^{-1}\left(\bigvee C\right)\right)<\infty$, we have $\left(\bigvee C\right)^{-1}\beta_{s,t}\neq e$. Firstly, if $s^{-1}\left(\bigvee C\right)< \left(s^{-1}\left(\bigvee C\right)\right)\vee \left(t^{-1}\left(\bigvee C\right)\right)$, then 
\begin{align*}
\big(\bigvee C\big)^{-1}\beta_{s,t}
&=\big(\bigvee C\big)^{-1}s\big(\big(s^{-1}\big(\bigvee C\big)\big)\vee \big(t^{-1}\big(\bigvee C\big)\big)\big)\\
&=\big(s^{-1}\big(\bigvee C\big)\big)\big(\big(s^{-1}\big(\bigvee C\big)\big)\vee \big(t^{-1}\big(\bigvee C\big)\big)\big)
\neq e.
\end{align*}
On the other hand, if $s^{-1}\left(\bigvee C\right)=\left(s^{-1}\left(\bigvee C\right)\right)\vee \left(t^{-1}\left(\bigvee C\right)\right)$, then
\[
\big(\bigvee C\big)^{-1}\beta_{s,t}=\big(\bigvee C\big)^{-1}ts^{-1}\big(\bigvee C\big)\neq e
\]
since $s\neq t$. Thus, $\psi_e |_\mathcal{K}$ is faithful. Hence, by Proposition~\ref{existence of rho map} it follows that the subspace $\mathcal{M}:=\overline{\psi_{\bigvee C}(\mathbf{X}_{\bigvee C})\mathcal{K}}$ is $\rho_{\bigvee C}^\psi$-invariant and $\rho_{\bigvee C}^\psi|_\mathcal{M}$ is faithful. To see that the map $\mathcal{L}_A\left(\mathbf{X}_{\bigvee C}\right)\ni T \mapsto R_{C,F}^\psi  \rho_{\bigvee C}^\psi(T)\in \mathcal{B}(\mathcal{H})$ is faithful, it remains to show that $\mathcal{M}\subseteq R_{C,F}^\psi\mathcal{H}$. Suppose $s,t\in C$ with $s\neq t$ and $\left(s^{-1}\left(\bigvee C\right)\right)\vee \left(t^{-1}\left(\bigvee C\right)\right)<\infty$. Since $\bigvee C\leq \beta_{s,t}$, Lemma~\ref{projections and psi} tells us that for any $x\in \mathbf{X}_{\bigvee C}$ we have 
\[
P_{\beta_{s,t}}\psi_{\bigvee C}(x)=\psi_{\bigvee C}(x)P_{\left(\bigvee C\right)^{-1}\beta_{s,t}}.
\]
Thus,
\begin{align*}
R_{C,F}^\psi&\psi_{\bigvee C}(x)\\
&=
Q_{C,F}^\psi\prod_{\substack{s,t\in C, \, s\neq t, \\ \left(s^{-1}\left(\bigvee C\right)\right)\vee \left(t^{-1}\left(\bigvee C\right)\right)<\infty}} \left(P^\psi_{\bigvee C}-P^\psi_{\beta_{s,t}}\right)\psi_{\bigvee C}(x)\\
&=
Q_{C,F}^\psi\psi_{\bigvee C}(x)\prod_{\substack{s,t\in C, \, s\neq t, \\ \left(s^{-1}\left(\bigvee C\right)\right)\vee \left(t^{-1}\left(\bigvee C\right)\right)<\infty}} \left(\mathrm{id}_\mathcal{H}-P^\psi_{\left(\bigvee C\right)^{-1}\beta_{s,t}}\right)\\
&=
\psi_{\bigvee C}(x)\prod_{\substack{p\in F\setminus C, \,  s,t\in C,\\ p\vee \left(\bigvee C\right) <\infty, \, s\neq t,\\ \left(s^{-1}\left(\bigvee C\right)\right)\vee \left(t^{-1}\left(\bigvee C\right)\right)<\infty}}\left(\mathrm{id}_\mathcal{H}-P_{\left(\bigvee C\right)^{-1}\left(p\vee \left(\bigvee C\right)\right)}^\psi\right)
\left(\mathrm{id}_\mathcal{H}-P^\psi_{\left(\bigvee C\right)^{-1}\beta_{s,t}}\right).
\end{align*}
Hence, $R_{C,F}^\psi$ is the identity on 
\begin{align*}
\mathcal{M}
&=
\overline{\psi_{\bigvee C}\left(\mathbf{X}_{\bigvee C}\right)\mathcal{K}}\\
&=
\overline{\psi_{\bigvee C}\left(\mathbf{X}_{\bigvee C}\right)
\hspace{-0.9em}
\prod_{\substack{p\in F\setminus C, \,  s,t\in C,\\ p\vee \left(\bigvee C\right) <\infty, \, s\neq t,\\ \left(s^{-1}\left(\bigvee C\right)\right)\vee \left(t^{-1}\left(\bigvee C\right)\right)<\infty}}
\hspace{-0.9em}
\left(\mathrm{id}_\mathcal{H}-P_{\left(\bigvee C\right)^{-1}\left(p\vee \left(\bigvee C\right)\right)}^\psi\right)
\left(\mathrm{id}_\mathcal{H}-P^\psi_{\left(\bigvee C\right)^{-1}\beta_{s,t}}\right)\mathcal{H}},
\end{align*}
and so $\mathcal{M}\subseteq R_{C,F}^\psi\mathcal{H}$. 

Putting all of this together, we see that
\begin{align*}
\bigg\|\sum_{k:p_k=q_k}\psi_{p_k}(x_k)\psi_{p_k}(y_k)^*\bigg\|_{\mathcal{B}(\mathcal{H})}
&=\bigg\|\sum_{k:p_k=q_k}\iota_{p_k}^{\bigvee C}(\Theta_{x_k,y_k})\bigg\|_{\mathcal{L}_A\left(\mathbf{X}_{\bigvee C}\right)}\\
&=\bigg\|R_{C,F}^\psi\rho_{\bigvee C}^\psi\bigg(\sum_{k:p_k=q_k}\iota_{p_k}^{\bigvee C}(\Theta_{x_k,y_k})\bigg)\bigg\|_{\mathcal{B}(\mathcal{H})}\\
&=\bigg\|R_{C,F}^\psi\sum_{k:p_k=q_k}\psi_{p_k}(x_k)\psi_{p_k}(y_k)^*\bigg\|_{\mathcal{B}(\mathcal{H})}\\
&=\bigg\|R_{C,F}^\psi\bigg(\sum_k \psi_{p_k}(x_k)\psi_{q_k}(y_k)^*\bigg)R_{C,F}^\psi\bigg\|_{\mathcal{B}(\mathcal{H})}\\
&\leq \bigg\|\sum_k \psi_{p_k}(x_k)\psi_{q_k}(y_k)^*\bigg\|_{\mathcal{B}(\mathcal{H})}.
\end{align*} 
Since the norm estimate (\ref{norm estimate onto core}) holds, the formula $\psi_p(x)\psi_q(y)^*\mapsto \delta_{p,q}\psi_p(x)\psi_q(y)^*$ extends to a map on $\psi_*(\mathcal{NT}_\mathbf{X})=\cspan\left\{\psi_p(x)\psi_q(y)^*:p,q\in P, \ x\in \mathbf{X}_p, \ y \in \mathbf{X}_q\right\}$ by linearity and continuity, which we denote by $E_\psi$. Furthermore, for any $p,q\in P$, $x\in \mathbf{X}_p$, $y \in \mathbf{X}_q$, we have
\begin{align*}
\left(E_\psi \circ \psi_*\right)\left(i_{\mathbf{X}_p}(x)i_{\mathbf{X}_q}(y)^*\right)
&=E_\psi\left(\psi_p(x)\psi_q(y)^*\right)
=\delta_{p,q}\psi_p(x)\psi_q(y)^*\\
&=\psi_*\left(\delta_{p,q} i_{\mathbf{X}_p}(x)i_{\mathbf{X}_q}(y)^*\right)
=\left(\psi_* \circ E_{\delta_\mathbf{X}}\right)\left(i_{\mathbf{X}_p}(x)i_{\mathbf{X}_q}(y)^*\right). 
\end{align*}
Since $\mathcal{NT}_\mathbf{X}=\cspan\left\{i_{\mathbf{X}_p}(x)i_{\mathbf{X}_q}(y)^*:p,q\in P, x\in \mathbf{X}_p, y \in \mathbf{X}_q\right\}$, whilst the maps $E_\psi \circ \psi_*$ and $\psi_* \circ E_{\delta_\mathbf{X}}$ are linear and norm-decreasing, we conclude that $E_\psi \circ \psi_*=\psi_* \circ E_{\delta_\mathbf{X}}$. This completes the proof of part (ii).
\end{proof}
\end{prop}

Finally, we prove the uniqueness theorem for Nica--Toeplitz algebras. We remind the reader that if $(G,P)$ is a quasi-lattice ordered group with $G$ amenable, then any compactly aligned product system $\mathbf{X}$ over $P$ is automatically amenable (in the sense that the expectation $E_{\delta_\mathbf{X}}$ is faithful on positive elements). 

\begin{proof}[Proof of Theorem~\ref{uniqueness theorem for NT algebras}]
Suppose that the representation
\begin{align*}
A\ni a \mapsto \psi_e(a)\prod_{t\in K}\big(\mathrm{id}_\mathcal{H}-P_t^\psi\big)\in \mathcal{B}(\mathcal{H})
\end{align*}
is faithful for any finite set $K\subseteq P\setminus \{e\}$. Let $b\in \mathcal{NT}_\mathbf{X}$ be such that $\psi_*(b)=0$. Thus,
\[
\psi_*(E_{\delta_\mathbf{X}}(b^*b))=E_\psi(\psi_*(b^*b))=E_\psi(\psi_*(b)^*\psi_*(b))=0.
\]
Since $\psi_*$ is faithful on $\mathcal{NT}_\mathbf{X}^{\delta_\mathbf{X}}=E_{\delta_\mathbf{X}}(\mathcal{NT}_\mathbf{X})$, we must have $E_{\delta_\mathbf{X}}(b^*b)=0$. As $E_{\delta_\mathbf{X}}$ is faithful on positive elements, we conclude that $b=0$. Hence, $\psi_*$ is faithful.

We now prove (ii). Suppose that $\psi_*$ is faithful and $\phi_p(A)\subseteq \mathcal{K}_A(\mathbf{X}_p)$ for each $p\in P$. Fix a finite set $K\subseteq P\setminus \{e\}$. For each $a\in A$, define
\[
T_a:=\sum_{J\subseteq K, \, \bigvee J<\infty}(-1)^{|J|}i_\mathbf{X}^{(\bigvee J)}\left(\phi_{\bigvee J}(a)\right)\in \mathcal{NT}_\mathbf{X}.
\]
We claim that for any Nica covariant representation $\varphi:\mathbf{X}\rightarrow \mathcal{B}(\mathcal{H}')$ we have 
\begin{equation}
\label{something about representations}
\begin{aligned}
\varphi_*(T_a)=\varphi_e(a)\prod_{t\in K}\big(\mathrm{id}_{\mathcal{H}'}-P_t^\varphi\big). 
\end{aligned}
\end{equation}
To see this, firstly observe that for any $t\in P$ and $a\in A$, we have
\[
\varphi_*\big(i_\mathbf{X}^{(t)}(\phi_t(a))\big)=\rho_t^\varphi(\phi_p(a))=\varphi_e(a)P_t^\varphi.
\]
Therefore,
\begin{align*}
\varphi_*(T_a)
&=\sum_{J\subseteq K, \, \bigvee J<\infty}(-1)^{|J|}\varphi_*\big(i_\mathbf{X}^{(\bigvee J)}\left(\phi_{\bigvee J}(a)\right)\big)\\
&=\varphi_e(a)\sum_{J\subseteq K, \, \bigvee J<\infty}(-1)^{|J|}P_{\bigvee J}^\varphi\\
&=\varphi_e(a)\sum_{J\subseteq K}(-1)^{|J|}P_{\bigvee J}^\varphi.
\end{align*}
Hence, to prove that Equation~(\ref{something about representations}) holds, it suffices to show that 
\begin{equation}
\label{realising product as sum}
\begin{aligned}
\sum_{J\subseteq K}(-1)^{|J|}P_{\bigvee J}^\varphi=\prod_{t\in K}\big(\mathrm{id}_{\mathcal{H}'}-P_t^\varphi\big). 
\end{aligned}
\end{equation}
To prove this we use induction on $|K|$. When $|K|=0$ we have
\begin{align*}
\sum_{J\subseteq K}(-1)^{|J|}P_{\bigvee J}^\varphi
=(-1)^{|\emptyset|}P_{\bigvee \emptyset}^\varphi
=P_e^\varphi
=\mathrm{id}_{\mathcal{H}'}
=\prod_{t\in \emptyset}\left(\mathrm{id}_{\mathcal{H}'}-P_t^\varphi\right)
=\prod_{t\in K}\left(\mathrm{id}_{\mathcal{H}'}-P_t^\varphi\right). 
\end{align*}
Now let $n\in \N$ and suppose we have equality whenever $K\subseteq P$ and $|K|=n$. Fix $K'\subseteq P$ with $|K'|=n+1$. Let $s\in K'$. Then
\begin{align*}
\sum_{J\subseteq K'}(-1)^{|J|}P_{\bigvee J}^\varphi
&=\sum_{J\subseteq K'\setminus \{s\}}(-1)^{|J|}P_{\bigvee J}^\varphi+\sum_{\{J\subseteq K': s\in J\}}(-1)^{|J|}P_{\bigvee J}^\varphi\\
&=\sum_{J\subseteq K'\setminus \{s\}}(-1)^{|J|}P_{\bigvee J}^\varphi+\sum_{J\subseteq K'\setminus \{s\}}(-1)^{|J\cup \{s\}|}P_{\bigvee (J\cup \{s\})}^\varphi\\
&=\sum_{J\subseteq K'\setminus \{s\}}(-1)^{|J|}P_{\bigvee J}^\varphi-\sum_{J\subseteq K'\setminus \{s\}}(-1)^{|J|}P_{\bigvee J}^\varphi P_{s}^\varphi\\
&=\left(\mathrm{id}_{\mathcal{H}'}-P_{s}^\varphi\right)\sum_{J\subseteq K'\setminus \{s\}}(-1)^{|J|}P_{\bigvee J}^\varphi\\
&=\left(\mathrm{id}_{\mathcal{H}'}-P_{s}^\varphi\right)\prod_{t\in  K'\setminus \{s\}}\left(\mathrm{id}_{\mathcal{H}'}-P_t^\varphi\right)\\
&=\prod_{t\in  K'}\left(\mathrm{id}_{\mathcal{H}'}-P_t^\varphi\right).
\end{align*}
This proves that Equation~(\ref{realising product as sum}) holds, and so Equation~(\ref{something about representations}) follows. 

Now let $\pi:A\rightarrow \mathcal{B}(\mathcal{H}')$ be a faithful nondegenerate representation of $A$. Define $\Psi:\mathbf{X}\rightarrow \mathcal{B}(\mathcal{F}_\mathbf{X}\otimes_A \mathcal{H}')$ by
\[
\Psi:=\big(\mathcal{F}_\mathbf{X}\text{-}\mathrm{Ind}_A^{\mathcal{L}_A(\mathcal{F}_\mathbf{X})}\pi\big)\circ l
\]
where $l:\mathbf{X}\rightarrow \mathcal{L}_A(\mathcal{F}_\mathbf{X})$ is the Fock representation of $\mathbf{X}$.  Since $l$ is a Nica covariant representation of $\mathbf{X}$ and $\mathcal{F}_\mathbf{X}\text{-}\mathrm{Ind}_A^{\mathcal{L}_A(\mathcal{F}_\mathbf{X})}\pi$ is a $*$-homomorphism, $\Psi$ is a Nica covariant representation of $\mathbf{X}$. We claim that the representation 
\[
A\ni a \mapsto \Psi_e(a)\prod_{t\in K}\big(\mathrm{id}_{\mathcal{F}_\mathbf{X}\otimes_A \mathcal{H}'}-P_t^\Psi\big)\in \mathcal{B}(\mathcal{F}_\mathbf{X}\otimes_A \mathcal{H}')
\]
is faithful. To see this, suppose that $a\in A\setminus\{0\}$, so that $aa^*\neq 0$. As $\pi$ is faithful, we can find $h\in \mathcal{H}'$ such that $\pi(aa^*)h\neq 0$. For any $t\in P\setminus \{e\}$ we have
\begin{align*}
\overline{\Psi_t(\mathbf{X}_t)\left(\mathcal{F}_\mathbf{X}\otimes_A \mathcal{H}'\right)}
&=\overline{\big(\mathcal{F}_\mathbf{X}\text{-}\mathrm{Ind}_A^{\mathcal{L}_A(\mathcal{F}_\mathbf{X})}\pi\left(l_t(\mathbf{X}_t)\right)\big)\left(\mathcal{F}_\mathbf{X}\otimes_A \mathcal{H}'\right)}\\
&=\overline{l_t(\mathbf{X}_t)(\mathcal{F}_\mathbf{X})\otimes_A \mathcal{H}'}\\
&=\overline{\bigoplus_{s\geq t}\mathbf{X}_{s}\otimes_A \mathcal{H}'}.
\end{align*}
Hence, it follows that $P_t^\Psi=\mathrm{proj}_{\overline{\Psi_t(\mathbf{X}t)\left(\mathcal{F}_\mathbf{X}\otimes_A \mathcal{H}'\right)}}$ is zero on $A\otimes_A \mathcal{H}'\subseteq \mathcal{F}_\mathbf{X}\otimes_A \mathcal{H}'$. Thus, as $e\not \in K$ we see that
\begin{align*}
\Psi_e(a)\prod_{t\in K}\left(\mathrm{id}_{\mathcal{F}_\mathbf{X}\otimes_A \mathcal{H}'}-P_t^\Psi\right)\left(a^*\otimes_A h\right)
=\Psi_e(a)(a^*\otimes_A h)
=l_e(a)(a^*)\otimes_A h
&=aa^*\otimes_A h,
\end{align*}
which is nonzero because
\begin{align*}
\norm{aa^*\otimes_A h}_{\mathcal{F}_\mathbf{X}\otimes_A \mathcal{H}'}^2
&=\langle aa^*\otimes_A h, aa^*\otimes_A h \rangle_\C
=\langle h, \pi((aa^*)^*aa^*)h \rangle_\C\\
&=\langle \pi(aa^*)h, \pi(aa^*)h \rangle_\C
=\norm{ \pi(aa^*)h}_{\mathcal{H}'}^2
\neq 0. 
\end{align*}
Therefore, $A\ni a \mapsto \Psi_e(a)\prod_{t\in K}\left(\mathrm{id}_{\mathcal{F}_\mathbf{X}\otimes_A \mathcal{H}'}-P_t^\Psi\right)\in \mathcal{B}(\mathcal{F}_\mathbf{X}\otimes_A \mathcal{H}')$ is faithful. 

Putting all of this together, and using that $\psi_*$ is faithful at the penultimate equality, we see that for any $a\in A$, 
\begin{align*}
\norm{a}_A
=\bigg\|\Psi_e(a)\prod_{t\in K}\big(\mathrm{id}_{\mathcal{F}_\mathbf{X}\otimes_A \mathcal{H}'}-P_t^\Psi\big)\bigg\|_{\mathcal{B}(\mathcal{F}_\mathbf{X}\otimes_A \mathcal{H}')}
&=\norm{\Psi_*(T_a)}_{\mathcal{B}(\mathcal{F}_\mathbf{X}\otimes_A \mathcal{H}')}\\
&\leq \norm{T_a}_{\mathcal{NT}_\mathbf{X}}\\
&=\norm{\psi_*(T_a)}_{\mathcal{B}(\mathcal{H})}\\
&=\bigg\|\psi_e(a)\prod_{t\in K}\big(\mathrm{id}_\mathcal{H}-P_t^\psi\big)\bigg\|_{\mathcal{B}(\mathcal{H})}.
\end{align*}
Hence, $A\ni a \mapsto \psi_e(a)\prod_{t\in K}\big(\mathrm{id}_\mathcal{H}-P_t^\psi\big)\in \mathcal{B}(\mathcal{H})$ is faithful. 
\end{proof}

In practice, we are often interested in representations of product systems in more general $C^*$-algebras, rather than on Hilbert spaces. The following corollary shows that provided the coefficient algebra acts compactly on each fibre of the product system, we can still characterise the faithfulness of the induced representation. 

\begin{cor}
\label{representations in C*-algebras}
Let $(G,P)$ be a quasi-lattice ordered group and $\mathbf{X}$ an amenable compactly aligned product system over $P$ with coefficient algebra $A$. Suppose that $A$ acts compactly on each $\mathbf{X}_p$. Let $\psi:\mathbf{X}\rightarrow B$ be a Nica covariant representation of $\mathbf{X}$ in a $C^*$-algebra $B$. Then the induced $*$-homomorphism $\psi_*:\mathcal{NT}_\mathbf{X}\rightarrow B$ is faithful if and only if for every $a\in A\setminus \{0\}$ and every finite set $K\subseteq P\setminus \{e\}$, we have
\[
\prod_{t\in K}\big(\psi_e-\psi^{(t)}\circ\phi_t\big)(a)\neq 0.
\]
\begin{proof}
Fix a faithful representation $\pi:B\rightarrow \mathcal{B}(\mathcal{H})$. Then $\pi\circ \psi:\mathbf{X}\rightarrow \mathcal{B}(\mathcal{H})$ is a Nica covariant representation with induced representation $(\pi\circ \psi)_*=\pi\circ \psi_*$. If $a\in A$ and $t\in P\setminus\{e\}$, then Proposition~\ref{existence of rho map} implies that 
\[
\pi(\psi_e(a))P_t^{\pi\circ \psi}=\rho_t^{\pi\circ \psi}(\phi_t(a))=(\pi\circ \psi)^{(t)}(\phi_t(a))=\pi\big(\psi^{(t)}(\phi_t(a))\big).
\]
Hence, for any  $a\in A$ and any finite set $K\subseteq P\setminus \{e\}$, we have that
\begin{align*}
\pi\bigg(\prod_{t\in K}\big(\psi_e-\psi^{(t)}\circ\phi_t\big)(a)\bigg)
&=\prod_{t\in K}\Big(\pi(\psi_e(a))-\pi\big(\psi^{(t)}(\phi_t(a))\big)\Big)\\
&=\prod_{t\in K}\Big(\pi(\psi_e(a))-\pi(\psi_e(a))P_t^{\pi\circ \psi}\Big)\\
&=(\pi\circ \psi)_e(a)\prod_{t\in K}\big(\mathrm{id}_\mathcal{H}-P_t^{\pi\circ \psi}\big),
\end{align*}
and so the result follows from Theorem~\ref{uniqueness theorem for NT algebras}.
\end{proof}
\end{cor}

\section{Acknowledgements}
The results in this article are from my PhD thesis. Thank you to my supervisors Adam Rennie and Aidan Sims at the University of Wollongong for their advice and encouragement during my PhD and during the writing of this article.

\printbibliography

\end{document}